\documentclass[12 pt]{article}
\usepackage{amssymb}
\usepackage{amsmath}
\usepackage{amsthm}
\allowdisplaybreaks
\usepackage{amscd}
\usepackage{graphicx}
\usepackage[all,cmtip]{xy}

\usepackage{geometry}

\usepackage{ifpdf}
\ifpdf
 \usepackage[colorlinks=true,linkcolor=blue,final,backref=page,hyperindex]{hyperref}
\else
  \usepackage[colorlinks,final,backref=page,hyperindex,hypertex]{hyperref}
\fi
\usepackage{txfonts}

\def\vbar{\mathchoice{\vrule height6.3ptdepth-.5ptwidth.8pt\kern- .8pt}
{\vrule height6.3ptdepth-.5ptwidth.8pt\kern-.8pt} {\vrule
height4.1ptdepth-.35ptwidth.6pt\kern-.6pt} {\vrule
height3.1ptdepth-.25ptwidth.5pt\kern-.5pt}}

\def\<{\langle}
\def\>{\rangle}

\newtheorem{thm}{Theorem}[section]

\newtheorem{cor}[thm]{Corollary}
\newtheorem{pro}[thm]{Proposition}
\newtheorem{ex}[thm]{Example}

\theoremstyle{definition}
\newtheorem{defi}{Definition}[section]

\theoremstyle{remark}
\newtheorem{rmk}{Remark}[section]
\begin{document}
\title{(Co)homology of compatible associative algebras}
\author{
{ Taoufik  Chtioui$^{1}$\footnote { Corresponding author,  E-mail: chtioui.taoufik@yahoo.fr} ,
Apurba Das$^{2}$\footnote {Corresponding author,  E-mail:  apurbadas348@gmail.com}
~~and~ Sami Mabrouk$^{3}$\footnote { Corresponding author,  E-mail: mabrouksami00@yahoo.fr}
}\\
{\small 1.  University of Sfax, Faculty of Sciences Sfax,  BP
1171, 3038 Sfax, Tunisia} \\
{\small 2.  Department of Mathematics and Statistics, Indian Institute of Technology}\\
{\small Kanpur 208016, Uttar Pradesh, India }\\
{\small 3.~ University of Gafsa, Faculty of Sciences Gafsa, 2112 Gafsa, Tunisia}}
\date{}
\maketitle

\begin{abstract}
In this paper, we define and study (co)homology theories of a compatible associative algebra $A$. At first, we construct a new graded Lie algebra whose Maurer-Cartan elements are given by compatible associative structures. Then we define the cohomology of a compatible associative algebra $A$ and as applications, we study extensions, deformations and extensibility of finite order deformations of $A$. We end this paper by considering compatible presimplicial vector spaces and the homology of compatible associative algebras.
\end{abstract}

\noindent \textbf{Key words}: Compatible associative algebras, (Co)homology, Extensions, Deformations.

\noindent \textbf{Mathematics Subject Classification 2020}: 16E40, 16S80.


\numberwithin{equation}{section}

\section*{Introduction}
Two algebraic structures of the same kind in a linear category are said to be compatible if their sum also defines the same kind of algebraic structure. Compatible structures often appear in various fields of mathematics and mathematical physics. Among others, the notion compatible Lie algebras are closely related to linear deformations (in particular, deformations by Nijenhuis operators) of Lie algebras and classical Yang-Baxter equations \cite{nij-ric,golu3}. They also appeared in the study of principal chiral fields \cite{golu1}, loop algebras over Lie algebras \cite{golu2} and elliptic theta functions \cite{Odesskii1}. In the mathematical study of biHamiltonian mechanics, compatible Poisson structures first appeared in the work of Magri, Morosi and Schwarzbach \cite{mag-mor,kos}. There is a close connection between compatible Lie algebras and compatible Poisson structures via dualization \cite{bol}. In the same spirit, compatible associative algebras are introduced and widely studied \cite{Odesskii2,Odesskii3}. Note that a compatible associative algebra is a triple $A= (A, \mu_1, \mu_2)$ such that $(A, \mu_1)$, $(A, \mu_2)$ and $(A, \mu_1 + \mu_2)$ are all associative algebras. See Section \ref{section-ca} for more details. The relation between compatible associative algebras and associative Yang-Baxter equations, quiver representations, bialgebra theory are explored in \cite{mar,Odesskii2,Odesskii3,wu}. See \cite{dotsenko,stro} for study on compatible algebras from operadic points of view.

\medskip

Cohomology and homology are some invariants for algebraic structures that begins with the works of Hochschild, Harrison, Barr among others \cite{hoch,harr,barr}. In \cite{gers} Gerstenhaber developed the pioneer theory of formal deformations for associative algebras and subsequently generalized for Lie algebras by Nijenhuis and Richardson \cite{nij-ric}. Such deformations are governed by the cohomology of algebras. Later, Balavoine \cite{bala} generalized the results of Gerstenhaber and Nijenhuis-Richardson to algebras over any quadratic operads. On the other hand, the homology of algebras is useful to study K\"{a}hler differentials and differential forms on algebras. Recently, the authors in \cite{comp-lie} defined a cohomology theory for compatible algebras and study linear deformations of compatible Lie algebras. Their study relies on the construction of a so-called bidifferential graded Lie algebra whose Maurer-Cartan elements are compatible Lie algebra structures. This bidifferential graded Lie algebra is not far from the Nijenhuis-Richardson graded Lie algebra constructed in \cite{nij-ric} to study Lie algebras. Therefore, this bidifferential graded Lie algebra has some lack of information to study compatible Lie algebras which led the authors of \cite{comp-lie} not to study extensions of finite order deformations.

\medskip

Our first aim is to construct a graded Lie algebra suitable for compatible algebraic structures. In this paper, we mainly focus on compatible associative algebras. Compatible algebras over any quadratic operads will be treated elsewhere. Here we define a graded Lie algebra (using the Gerstenhaber bracket \cite{gers2}) whose Maurer-Cartan elements are given by compatible associative algebras. Then we construct the cohomology of a compatible associative algebra $(A, \mu_1, \mu_2)$ with coefficients in a suitable bimodule. We show that the cohomology with coefficients in itself can be seen as the cohomology induced by the corresponding Maurer-Cartan element. We note that our cohomology of compatible associative algebras is not a combination of cohomologies of $(A, \mu_1)$ and $(A, \mu_2)$. However, we observe that there is a morphism from the cohomology of compatible associative algebra $(A, \mu_1, \mu_2)$ to the cohomology of the associative algebra $(A, \mu_1 + \mu_2)$. As applications of cohomology, we study extensions and various types of deformations (e.g., linear, formal and finite order) of a compatible associative algebra. During the course, we introduce Nijenhuis operators that induce trivial linear deformations. We also show that the vanishing of the second cohomology group of a compatible associative algebra $A$ implies that $A$ is rigid. Moreover, the obstruction to extending a finite order deformation is always a $3$-cocycle. Thus, vanishing of the third cohomology group implies that any finite order deformations are extensible.

\medskip

In the next, we also introduce homology for compatible associative algebras with coefficients in a suitable bimodule. To do this, we first define a notion of compatible presimplicial module and show that a compatible presimplicial module induces a homology. Like cohomology, the homology of a compatible associative algebra $(A, \mu_1, \mu_2)$ is not a combination of the homologies of $(A, \mu_1)$ and $(A, \mu_2)$.

\medskip

The paper is organized as follows. In Section \ref{section:background}, we recall the (Hochschild) cohomology of associative algebras and the Gerstenhaber bracket. In the next section (Section \ref{section-ca}) we consider compatible associative algebras and compatible bimodules. We also construct the graded Lie algebra whose Maurer-Cartan elements are precisely compatible associative structures. In Section \ref{sec:cohomology}, we introduce the cohomology of a compatible associative algebra with coefficients in a compatible bimodule. Abelian extensions of compatible associative algebras are also treated. In Section \ref{section:deformation}, we study deformations of compatible associative algebras from cohomological points of view. Finally, compatible presimplicial modules and homology of compatible associative algebras are given in Section \ref{section:homology}. We end this paper by mentioning some problems of interest regarding (co)homology of compatible associative algebras.

\subsection*{Notations}
Given an associative algebra $(A, \mu)$, we use the notation $a \cdot b$ for the element $\mu (a, b),$ for $a,b \in A$. If $(M, l, r)$ is a bimodule over the associative algebra $A$, we use the same notation dot for left and right $A$-actions on $M$, i.e., we write $a \cdot m$ for the element $l(a,m)$ and $m \cdot a$ for the element $r (m,a)$. For a compatible associative algebra $(A, \mu_1, \mu_2)$, we write $a \cdot_1 b$ for $\mu_1 (a, b)$ and $a \cdot_2 b$ for $\mu_2 (a, b)$. Similar notations are used for compatible bimodules over a compatible associative algebra.

All vector spaces, (multi)linear maps and tensor products are over a field $\mathbb{K}$ of characteristic $0$. All vector spaces are finite-dimensional.

\section{Background} \label{section:background}
In this section, we recall the (Hochschild) cohomology theory of associative algebras and the Gerstenhaber bracket. Our main references are \cite{gers2,gers,loday-book}.

\subsection{Cohomology of associative algebras}

Let $(A, \mu)$ be an associative algebra, i.e., $A$ is a vector space and $\mu : A \otimes A \rightarrow A, (a, b) \mapsto a \cdot b$ is a bilinear map satisfying the associativity condition
\begin{align*}
( a \cdot b) \cdot c = a \cdot ( b \cdot c), ~\text{ for } a, b, c \in A.
\end{align*}

We may also denote an associative algebra simply by $A$ when the multiplication map is clear from the context.

\medskip

A bimodule over an associative algebra $A$ consists of a vector space $M$ together with bilinear maps (called left and right $A$-actions) $l : A \otimes M \rightarrow M, (a, m) \mapsto a \cdot m$ and $r : M \otimes A \rightarrow M, (m, a) \mapsto m \cdot a$ satisfying the following compatibilities
\begin{align*}
(a \cdot b ) \cdot m = a \cdot ( b \cdot m), \qquad ( a \cdot m ) \cdot b = a \cdot (m \cdot b) ~~~ \text{ and } ~~~ (m \cdot a) \cdot b = m \cdot ( a \cdot b),
\end{align*}
for $a, b \in A$ and $m \in M$. Note that we have used the same notation $\cdot$ for the multiplication on $A$ as well as left and right $A$-actions. It will be understood from the entries that which operation we mean. An $A$-bimodule as above may be simply denoted by $M$ when both left and right $A$-actions on $M$ are clear.

It follows that $A$ is an $A$-bimodule with both left and right $A$-actions are given by the algebra multiplication map.

Given an associative algebra $A$ and an $A$-bimodule $M$, the (Hochschild) cohomology groups can be defined as follows. For each $n \geq 0$, the $n$-th cochain group $C^n( A, M)$ is given by
\begin{align*}
C^n(A,M) : = \mathrm{Hom}(A^{\otimes n}, M),
\end{align*}
and the (Hochschild) coboundary map $\delta : C^n(A, M) \rightarrow C^{n+1} (A, M)$, for $n \geq 0$, is given by
\begin{align}\label{hoch-diff}
 (\delta f) (a_1, a_2, \ldots, a_{n+1}) =~& a_1 \cdot f(a_2, \ldots , a_{n+1}) + \sum_{i=1}^{n} (-1)^i f ( a_1, \ldots, a_{i-1}, a_i \cdot a_{i+1}, a_{i+2}, \ldots, a_{n+1} ) \\
~& + (-1)^{n+1} f (a_1, \ldots, a_n) \cdot a_{n+1}, \nonumber
\end{align}
for $f \in C^n(A, M)$ and $a_1, \ldots, a_{n+1} \in A$. The corresponding cohomology groups are called the Hochschild cohomology of $A$ with coefficients in the $A$-bimodule $M$.

\subsection{The Gerstenhaber bracket}\label{sec:GB}

Recall that, in \cite{gers2} Gerstenhaber construct a graded Lie algebra structure on the graded space of all multilinear maps on a vector space $A$. More precisely, he considered the graded space $C^{\ast + 1} (A, A) = C^\ast (A, A) [1] = \oplus_{n \geq 0} C^{n+1} (A, A)$ and defined a graded Lie bracket (called the Gerstenhaber bracket) on $C^{\ast +1 }(A,A)$ by
\begin{align}
[f, g]_G  =~& f \diamond g - (-1)^{mn}~ g \diamond f, ~~ \text{ where }  \label{gers-brk}\\
(f \diamond g)(a_1, \ldots, a_{m+n+1} ) =~&  \sum_{i = 1}^{m+1} (-1)^{(i-1)n}~f ( a_1, \ldots, a_{i-1}, g ( a_i, \ldots, a_{i+n}), \ldots, a_{m+n+1}), \nonumber
\end{align}
for $f \in C^{m+1} (A,A)$ and $g \in C^{n+1}(A,A)$. The importance of this graded Lie bracket is given by the following characterization of associative structures.

\begin{pro}
Let $A$ be a vector space and $\mu: A^{\otimes 2} \rightarrow A$ be a bilinear map on $A$. Then $\mu$ defines an associative structure on $A$ if and only if $[\mu,\mu]_G = 0$.
\end{pro}

Let $(A,\mu)$ be an associative algebra. Then it follows from (\ref{hoch-diff}) and (\ref{gers-brk}) that the Hochschild coboundary map $\delta : C^n (A,A) \rightarrow C^{n+1} (A,A)$ for the cohomology of $A$ with coefficients in itself is given by
\begin{align}\label{hoch-diff-brk}
\delta f = (-1)^{n-1} [\mu, f]_G, ~\text{ for } f \in C^n(A,A).
\end{align}

\section{Compatible associative algebras and their characterization}\label{section-ca}

In this section, we first recall compatible associative algebras and then define compatible bimodules over them. We also construct a graded Lie algebra whose Maurer-Cartan elements are compatible associative structures.

\subsection{Compatible associative algebras} In this subsection, we consider compatible associative algebras and compatible bimodules over them.

\begin{defi}
A compatible associative algebra is a triple $(A, \mu_1, \mu_2)$ in which $(A, \mu_1)$ and $(A, \mu_2)$ are both associative algebras satisfying the following compatibility
\begin{align*}
(a \cdot_1 b ) \cdot_2 c + ( a \cdot_2 b) \cdot_1 c = a \cdot_1 ( b \cdot_2 c ) + a \cdot_2 ( b \cdot_1 c ), ~ \text{ for } a, b, c \in A.
\end{align*}
Here $\cdot_1$ and $\cdot_2$ are used for the multiplications $\mu_1$ and $\mu_2$, respectively.
\end{defi}

We may denote a compatible associative algebra as above by $(A, \mu_1, \mu_2)$ or simply by $A$, and say that $(\mu_1, \mu_2)$ is a compatible associative algebra structure on $A$.

\begin{rmk}
It follows from the above definition that the sum $\mu_1 + \mu_2$ also defines an associative product on $A$. In other words, $(A, \mu_1 + \mu_2)$ is an associative algebra. In fact, one can show that $(A , k \mu_1 + l \mu_2)$ is an associative algebra, for any $k, l \in \mathbb{K}.$
\end{rmk}

\begin{defi}
Let $A = (A, \mu_1, \mu_2)$ and $A' = (A', \mu_1', \mu_2')$ be two compatible associative algebras. A morphism of compatible associative algebras from $A$ to $A'$ is a linear map $\phi : A \rightarrow A'$ satisfying $\phi \circ \mu_1 = \mu_1' \circ (\phi \otimes \phi)$ and $\phi \circ \mu_2 = \mu_2' \circ (\phi \otimes \phi)$.
\end{defi}

\begin{pro}
Let $A$ be a vector space. Then a pair $(\mu_1, \mu_2)$ of bilinear maps on $A$ defines a compatible associative algebra structure on $A$ if and only if
\begin{align*}
[\mu_1, \mu_1]_G = 0, \qquad [\mu_2, \mu_2]_G = 0 ~~~~ \text{ and } ~~~~ [\mu_1, \mu_2]_G = 0.
\end{align*}
\end{pro}

\begin{ex}
Let $(A, \mu)$ be an associative algebra. A Nijenhuis operator on $A$ is a linear map $N : A \rightarrow A$ satisfying
\begin{align*}
N(a) \cdot N(b) = N \big( N(a) \cdot b + a \cdot N(b) - N (a \cdot b) \big), \text{ for } a, b \in A.
\end{align*}
A Nijenhuis operator $N$ induces a new associative multiplication on $A$, denoted by $\mu_N : A \otimes A \rightarrow A, (a,b) \mapsto a \cdot_N b$ and it is defined by
\begin{align*}
a \cdot_N b := N(a) \cdot b + a \cdot N(b) - N (a \cdot b), ~ \text{ for } a, b \in A.
\end{align*}
Then it is easy to see that $(A, \mu, \mu_N)$ is a compatible associative algebra.
\end{ex}

\begin{ex}
Let $(A, \mu)$ be an associative algebra. Then the graded space $C^\ast (A,A) = \oplus_{n \geq 1} C^n (A, A)$ of Hochschild cochains of $A$ (with coefficients in itself) carries an associative cup-product \cite{gers2} given by
\begin{align*}
(f \smile_\mu g )(a_1, \ldots, a_{m+n}) = f ( a_1, \ldots, a_m) \cdot g ( a_{m+1}, \ldots, a_{m+n}),
\end{align*}
for $f \in C^m (A,A)$ and $g \in C^n(A,A)$. It is easy to see that if $(A, \mu_1, \mu_2)$ is a compatible associative algebra, then $(C^\ast (A,A), \smile_{\mu_1}, \smile_{\mu_2})$ is a compatible associative algebra.
\end{ex}

\begin{ex}
Let $(A, \mu)$ be an associative algebra and $M$ be an $A$-bimodule. If $f \in C^2 (A, M)$ is a Hochschild $2$-cocycle on $A$ with coefficients in the $A$-bimodule $M$, then $A \oplus M$ can be equipped with the $f$-twisted semidirect product associative algebra given by
\begin{align*}
(a, m) \cdot_f (b, n) = ( a \cdot b, a \cdot n + m \cdot b + f (a, b)), ~ \text{ for } (a, m), (b,n) \in A \oplus M.
\end{align*}
With this notation, it can be easily checked that  $(A \oplus M, \cdot_0 , \cdot_f)$ is a compatible associative algebra.
\end{ex}

The next class of examples come from compatible Rota-Baxter operators on associative algebras. See \cite{guo-book} for more on Rota-Baxter operators.

\begin{defi}
Let $(A, \mu)$ be an associative algebra. A Rota-Baxter operator on $A$ is a linear map $R : A \rightarrow A$ satisfying
\begin{align*}
R(a) \cdot R(b) = R \big( R(a) \cdot b + a \cdot R(b) \big), ~ \text{ for } a, b \in A.
\end{align*}
\end{defi}

A Rota-Baxter operator $R$ induces a new associative multiplication $\mu_R : A \otimes A \rightarrow A,~(a,b) \mapsto a \cdot_R b$ on the underlying vector space $A$ and it is given by
\begin{align*}
a \cdot_R b = R(a) \cdot b + a \cdot R(b), ~ \text{ for } a, b \in A.
\end{align*}

\begin{defi}
Two Rota-Baxter operators $R, S : A \rightarrow A$ on an associative algebra $A$ are said to be compatible if for any $k,l \in \mathbb{K}$, the sum $k R + l S : A \rightarrow A$ is a Rota-Baxter operator on $A$. Equivalently,
\begin{align*}
R(a) \cdot S (b) + S(a) \cdot R (b) =  R \big( S(a) \cdot b + a \cdot S(b) \big) + S \big( R(a) \cdot b + a \cdot R(b) \big), ~ \text{ for } a, b \in A.
\end{align*}
\end{defi}

The following result is straightforward.

\begin{pro}
Let $R, S : A \rightarrow A$ be two compatible Rota-Baxter operators on $A$. Then $(A, \cdot_R )$ and $(A, \cdot_S)$ are compatible associative algebra structures on $A$.
\end{pro}

\medskip

\begin{defi}
Let $A = (A, \mu_1, \mu_2)$ be a compatible associative algebra. A compatible $A$-bimodule consists of a quintuple $(M, l_1, r_1, l_2, r_2)$ in which $M$ is a vector space and
\begin{align*}
\begin{cases}
l_1 : A \otimes M \rightarrow M,~ (a,m) \mapsto a \cdot_1 m, \\
r_1 : M \otimes A \rightarrow M,~ (m,a) \mapsto m \cdot_1 a,
\end{cases}
~~~~
\begin{cases}
l_2 : A \otimes M \rightarrow M,~ (a,m) \mapsto a \cdot_2 m, \\
r_2 : M \otimes A \rightarrow M,~ (m,a) \mapsto m \cdot_2 a,
\end{cases}
\end{align*}
are bilinear maps such that

-~ $(M, l_1, r_1)$ is a bimodule over the associative algebra $(A, \mu_1)$;

-~ $(M, l_2, r_2)$ is a bimodule over the associative algebra $(A, \mu_2)$;

-~ the following compatibilities are hold: for $a, b \in A$ and $m \in M$,
\begin{align}
( a \cdot_1 b ) \cdot_2 m  + ( a \cdot_2 b ) \cdot_1 m  = a \cdot_1 ( b \cdot_2 m) + a \cdot_2 ( b \cdot_1 m), \label{comp-bi1}\\
( a \cdot_1 m ) \cdot_2 b  + ( a \cdot_2 m ) \cdot_1 b  = a \cdot_1 ( m \cdot_2 b) + a \cdot_2 ( m \cdot_1 b), \label{comp-bi2}\\
( m \cdot_1 a ) \cdot_2 b  + ( m \cdot_2 a ) \cdot_1 b  = m \cdot_1 ( a \cdot_2 b) + m \cdot_2 ( a \cdot_1 b). \label{comp-bi3}
\end{align}
\end{defi}

A compatible $A$-bimodule as above may be simply denoted by $M$ when no confusion arises.

\begin{ex}
Any compatible associative algebra $A = (A, \mu_1, \mu_2)$ is a compatible $A$-bimodule in which $l_1 = r_1 = \mu_1$ and $l_2 = r_2 = \mu_2$.
\end{ex}

\begin{rmk}
Let $A = (A, \mu_1, \mu_2)$ be a compatible associative algebra and $(M, l_1, r_1, l_2, r_2)$ be a compatible $A$-bimodule. Then it is easy to see that $(M, l_1 + l_2, r_1 + r_2)$ is a bimodule over the associative algebra $(A, \mu_1 + \mu_2).$
\end{rmk}

Given an associative algebra and a bimodule, one can construct the dual bimodule \cite{loday-book}. This can be generalized to the case of compatible associative algebras.

\begin{pro}
Let $A$ be a compatible associative algebra and $M$ be a compatible $A$-bimodule. Then the dual space $M^*$ also carries a compatible $A$-bimodule structure given by
\begin{align*}
\begin{cases}
( a \cdot_1 \alpha) (m) = \alpha ( m \cdot_1 a), \\
(\alpha \cdot_1 a) (m) = \alpha ( a \cdot_1 m),
\end{cases}
~~~~
\begin{cases}
( a \cdot_2 \alpha) (m) = \alpha ( m \cdot_2 a), \\
(\alpha \cdot_2 a) (m) = \alpha ( a \cdot_2 m), \text{ for } a \in A, \alpha \in M^*, m \in M.
\end{cases}
\end{align*}
\end{pro}

\begin{proof}
We only need to check the compatibility conditions (\ref{comp-bi1})-(\ref{comp-bi3}). For $a, b \in A$, $\alpha \in M^*$ and $m \in M$, we have
\begin{align*}
\big( (a \cdot_1 b) \cdot_2 \alpha + ( a \cdot_2 b ) \cdot_1 \alpha   \big) (m) =~& \alpha \big( m \cdot_2 ( a \cdot_1 b) + m \cdot_1 ( a \cdot_2 b ) \big) \\
=~& \alpha \big(  (m \cdot_1 a) \cdot_2 b + (m \cdot_2 a) \cdot_1 b \big) \\
=~& (b \cdot_2 \alpha ) (m \cdot_1 a ) + (b \cdot_1 \alpha) (m \cdot_2 a) \\
=~& \big( a \cdot_1 ( b \cdot_2 \alpha ) + a \cdot_2 ( b \cdot_1 \alpha) \big) (m).
\end{align*}
Hence, we have $(a \cdot_1 b) \cdot_2 \alpha + ( a \cdot_2 b ) \cdot_1 \alpha   = a \cdot_1 ( b \cdot_2 \alpha ) + a \cdot_2 ( b \cdot_1 \alpha)$. This verifies the condition (\ref{comp-bi1}). The verifications of (\ref{comp-bi2}) and (\ref{comp-bi3}) are similar. Hence the proof.
\end{proof}

The following result generalizes the semidirect product for associative algebras \cite{loday-book}.

\begin{pro}\label{semi-prop}
Let $A$ be a compatible associative algebra and $M$ be a compatible $A$-bimodule. Then the direct sum $A \oplus M$ carries a compatible associative algebra structure given by
\begin{align*}
(a,m) \cdot_1 (b, n) = (a \cdot_1 b,~ a \cdot_1 n + m \cdot_1 b),\\
 (a,m) \cdot_2 (b, n) = (a \cdot_2 b,~ a \cdot_2 n + m \cdot_2 b),
\end{align*}
for $(a,m), (b,n) \in A \oplus M$. This is called the semidirect product.
\end{pro}

\subsection{A new graded Lie algebra and characterization of compatible associative algebras}

Let $(\mathfrak{g} = \oplus_n \mathfrak{g}^n, [~,~])$ be a graded Lie algebra. An element $\theta \in \mathfrak{g}^1$ is said to be a Maurer-Cartan element of $\mathfrak{g}$ if $\theta$ satisfies
\begin{align*}
[\theta, \theta] = 0.
\end{align*}

\begin{rmk}\label{mc-rem}
\begin{itemize}
\item[(i)] A Maurer-Cartan element $\theta$ induces a degree $1$ coboundary map $d_\theta : =[ \theta , - ]$ on $\mathfrak{g}$. Infact, the differential $d_\theta$ makes the triple $(\mathfrak{g}, [~,~], d_\theta)$ into a differential graded Lie algebra.

\item[(ii)] Let $\theta$ be a Maurer-Cartan element.  For any $\theta' \in \mathfrak{g}^1$, the sum $\theta + \theta'$ is a Maurer-Cartan element of $\mathfrak{g}$ if and only if $\theta'$ satisfies
\begin{align*}
d_\theta (\theta') + \frac{1}{2} [\theta', \theta' ] = 0.
\end{align*}
\end{itemize}
\end{rmk}

\begin{defi}
Two Maurer-Cartan elements $\theta_1$ and $\theta_2$ are said to be compatible if they additionally satisfy  $[\theta_1, \theta_2] = 0$. In this case, we say that $(\theta_1, \theta_2)$ is a compatible pair of Maurer-Cartan elements of $\mathfrak{g}$.
\end{defi}

In the rest of this section, we assume that the underlying graded vector space $\mathfrak{g}$ is non-negatively graded, i.e., $\mathfrak{g} = \oplus_{n \geq 0} \mathfrak{g}^n$. We construct a new graded Lie algebra $\mathfrak{g}_c$ whose Maurer-Cartan elements are precisely compatible pair of Maurer-Cartan elements of $\mathfrak{g}.$

We define $\mathfrak{g}_c = \oplus_{n \geq 0} (\mathfrak{g}_c)^n$, where
\begin{align*}
(\mathfrak{g}_c)^0 = \mathfrak{g}^0 ~~~~ \text{ and } ~~~~ (\mathfrak{g}_c)^n = \underbrace{\mathfrak{g}^n \oplus \cdots \oplus \mathfrak{g}^n}_{(n+1) \text{ times}}, ~\text{ for } n \geq 1.
\end{align*}
Let $\llbracket ~,~ \rrbracket : (\mathfrak{g}_c)^m \times (\mathfrak{g}_c)^n \rightarrow (\mathfrak{g}_c)^{m+n}$, for $m,n \geq 0$, be the degree $0$ bracket defined by
\begin{align}\label{br}
\llbracket &(f_1, \ldots, f_{m+1}), (g_1, \ldots, g_{n+1}) \rrbracket :=  \\
&= \big( [f_1, g_1],~ [f_1, g_2] + [f_2, g_1],~ \ldots, \underbrace{[f_1, g_i] + [f_2, g_{i-1}] + \cdots + [f_i, g_1]}_{i\text{-th place}},~ \ldots, [f_{m+1}, g_{n+1}] \big), \nonumber
\end{align}
for $(f_1, \ldots, f_{m+1}) \in (\mathfrak{g}_c)^m$ and $(g_1, \ldots, g_{n+1}) \in (\mathfrak{g}_c)^n.$

\begin{pro}
\begin{itemize}
\item[(i)] With the above notations, $(\mathfrak{g}_c, \llbracket ~, ~ \rrbracket)$ is a graded Lie algebra. Moreover, the map $\phi : \mathfrak{g}_c \rightarrow \mathfrak{g}$ defined by
\begin{align*}
\phi (f ) =~& f, ~ \text{ for } f \in (\mathfrak{g}_c)^0 = \mathfrak{g}^0,\\
\phi ((f_1, \ldots, f_{n+1})) =~& f_1 + \cdots + f_{n+1},~ \text{ for } (f_1, \ldots, f_{n+1}) \in (\mathfrak{g}_c)^n
\end{align*}
is a morphism of graded Lie algebras.
\item[(ii)] A pair $(\theta_1, \theta_2)$ of elements of $\mathfrak{g}^1$ is a compatible pair of Maurer-Cartan elements of $\mathfrak{g}$ if and only if $(\theta_1, \theta_2) \in (\mathfrak{g}_c)^1 = \mathfrak{g}^1 \oplus \mathfrak{g}^1$ is a Maurer-Cartan element in the graded Lie algebra $(\mathfrak{g}_c, \llbracket ~, ~ \rrbracket).$
\end{itemize}
\end{pro}

\begin{proof}
(i) For $(f_1, \ldots, f_{m+1}) \in (\mathfrak{g}_c)^m$, $(g_1, \ldots, g_{n+1}) \in (\mathfrak{g}_c)^n$ and $(h_1, \ldots, h_{p+1}) \in (\mathfrak{g}_c)^p$,
\begin{align*}
&\llbracket (f_1, \ldots, f_{m+1}), \llbracket (g_1, \ldots, g_{n+1}) ,  (h_1, \ldots, h_{p+1}) \rrbracket \rrbracket \\
&= \llbracket (f_1, \ldots, f_{m+1}) , \big( [g_1, h_1], \ldots, \underbrace{\sum_{q+r = i+1}  [g_q, h_r]}_{i\text{-th place}}, \ldots, [g_{n+1}, h_{p+1}] \big) \rrbracket \\
&= \big( [f_1, [g_1, h_1]], \ldots, \underbrace{ \sum_{p+q+r= i+2} [f_p, [g_q, h_r]]}_{i\text{-th place}}, \ldots, [f_{m+1}, [g_{n+1}, h_{p+1}]]      \big)\\
&= \bigg( [[f_1, g_1], h_1] + (-1)^{mn} ~ [g_1, [f_1, h_1]]~, \ldots, \underbrace{\sum_{p+q+r = i+2} [[f_p, g_q], h_r] + (-1)^{mn} ~[g_q, [f_p, h_r]]}_{i\text{-th place}}, \\
& \qquad \qquad \qquad \ldots,  [[f_{m+1}, g_{n+1}], h_{p+1}] + (-1)^{mn} ~ [g_{n+1}, [f_{m+1}, h_{p+1}]]    \bigg) \\
&= \llbracket  \llbracket (f_1, \ldots, f_{m+1}), (g_1, \ldots, g_{n+1}) \rrbracket, (h_1, \ldots, h_{p+1}) \rrbracket  \\
& \qquad \qquad \qquad + (-1)^{mn}~ \llbracket  (g_1, \ldots, g_{n+1}), \llbracket (f_1, \ldots, f_{m+1}), (h_1, \ldots, h_{n+1}) \rrbracket \rrbracket.
\end{align*}
Hence the first part follows. We also have
\begin{align*}
\phi \llbracket (f_1, \ldots, f_{m+1}), (g_1, \ldots, g_{n+1}) \rrbracket =~& \sum_{i=1}^{m+n+1} \sum_{q+r = i+1} [f_q, g_r] \\
=~&[f_1 + \cdots + f_{m+1}, ~g_1 + \cdots + g_{n+1}] \\
=~& [\phi (f_1, \ldots, f_{m+1}), \phi (g_1, \ldots, g_{n+1}) ],
\end{align*}
which completes the second part.

(ii) For a pair $(\theta_1, \theta_2)$ of elements of $\mathfrak{g}^1$, we have
\begin{align*}
\llbracket (\theta_1, \theta_2), (\theta_1, \theta_2) \rrbracket = ( [\theta_1, \theta_1], [\theta_1, \theta_2] + [\theta_2, \theta_1], [\theta_2, \theta_2]) = ([\theta_1, \theta_1], 2[\theta_1, \theta_2], [\theta_2, \theta_2]).
\end{align*}
Thus $(\theta_1, \theta_2) \in (\mathfrak{g}_c)^1$ is a Maurer-Cartan element in $\mathfrak{g}_c$ if and only if $(\theta_1, \theta_2)$ is a pair of compatible Maurer-Cartan elements in $\mathfrak{g}$.
\end{proof}

Thus, from the Gerstenhaber graded Lie bracket (defined in Subsection \ref{sec:GB}) and the above proposition, we get the following.

\begin{thm}
Let $A$ be a vector space.
\begin{itemize}
\item[(i)] Then $C^{\ast + 1}_c (A, A) := \oplus_{n \geq 0} C^{n+1}_c (A,A)$, where
\begin{align*}
C^1_c (A,A) = C^1(A,A) ~~~~ \text{ and } ~~~~ C^{n+1}_c (A,A) = \underbrace{C^{n+1} (A,A) \oplus \cdots \oplus C^{n+1}(A,A)}_{(n+1) \text{ times}}, ~\text{ for } n \geq 1
\end{align*}
is a graded Lie algebra with bracket given by (\ref{br}) where $[~,~]_G$ is replaced by $[~,~]$. Moreover, the map
\begin{align}\label{the-phi}
\phi : C^{\ast + 1}_c (A,A) \rightarrow C^{\ast + 1 } (A,A),~ (f_1, \ldots, f_{n+1}) \mapsto f_1 + \cdots + f_{n+1}, \text{ for } n \geq 0
\end{align}
is a morphism of graded Lie algebras.
\item[(ii)] A pair $(\mu_1, \mu_2) \in C^2(A,A) \oplus C^2(A,A)$ defines a compatible associative algebra structure on $A$ if and only if $(\mu_1, \mu_2) \in C^2_c(A,A)$ is a Maurer-Cartan element in the graded Lie algebra $(C^{\ast + 1}_c (A,A), \llbracket ~, ~ \rrbracket)$.
\end{itemize}
\end{thm}

Let $(A, \mu_1, \mu_2)$ be a compatible associative algebra. Then it follows from Remark \ref{mc-rem} that there is a degree $1$ coboundary map
\begin{align}\label{d-mu}
d_{(\mu_1, \mu_2)} := \llbracket (\mu_1, \mu_2), - \rrbracket : C^n_c (A,A) \rightarrow C^{n+1}_c (A,A), \text{ for } n \geq 1
\end{align}
which makes $(C^{\ast + 1}_c (A, A), \llbracket ~, ~ \rrbracket, d_{(\mu_1, \mu_2)} )$ into a differential graded Lie algebra.

\begin{rmk}
Later we will introduce the cohomology of a compatible associative algebra $A$ with coefficients in a compatible $A$-bimodule. We will see that the cohomology of the cochain complex $\{ C^\ast_c (A,A), d_{(\mu_1, \mu_2)} \}$ is the cohomology of $A$ with coefficients in itself.
\end{rmk}

We also have the following result from Remark \ref{mc-rem}.

\begin{pro}
Let $(\mu_1, \mu_2)$ be a compatible associative algebra structure on $A$. For any $(\mu_1', \mu_2') \in C^2_c (A,A) = C^2(A,A) \oplus C^2(A,A)$, the pair $(\mu_1 + \mu_1', \mu_2 + \mu_2')$ is a compatible associative algebra structure on $A$ if and only if $(\mu_1', \mu_2')$ satisfies
\begin{align*}
d_{(\mu_1, \mu_2)} (\mu_1', \mu_2') + \frac{1}{2} \llbracket (\mu_1', \mu_2') , (\mu_1', \mu_2')  \rrbracket = 0.
\end{align*}
\end{pro}

\section{Cohomology of compatible associative algebras}\label{sec:cohomology}
In this section, we introduce the cohomology of a compatible associative algebra with coefficients in a compatible bimodule. When considering the cohomology of a compatible associative algebra with coefficients in itself, it carries a graded Lie algebra structure. We also introduce abelian extensions of a compatible associative algebra and classify equivalence classes of abelian extensions in terms of the second cohomology group.

\subsection{Cohomology}

Let $(A, \mu_1, \mu_2)$ be a compatible associative algebra and $M = (M, l_1, r_1, l_2, r_2)$ be a compatible $A$-bimodule. Let
\begin{align*}
\delta_1 : C^n(A, M) \rightarrow C^{n+1} (A, M), ~ n \geq 0,
\end{align*}
denotes the coboundary operator for the Hochschild cohomology of $(A, \mu_1)$ with coefficients in the bimodule $(M, l_1, r_1)$, and
\begin{align*}
\delta_2 : C^n(A, M) \rightarrow C^{n+1} (A, M), ~ n \geq 0,
\end{align*}
denotes the coboundary operator for the Hochschild cohomology of $(A, \mu_2)$ with coefficients in the bimodule $(M, l_2, r_2)$. Then we obviously have
\begin{align*}
(\delta_1)^2 = 0 \qquad \text{ and } \qquad (\delta_2)^2 = 0.
\end{align*}
Since the two associative structures $\mu_1$ and $\mu_2$ on $A$, and the corresponding bimodule structures on $M$ are compatible, we may expect some compatibility between the coboundaries $\delta_1$ and $\delta_2$. Before we state the compatibility, we observe the following.

Here we first give the interpretation of $\delta_1$ and $\delta_2$ in terms of two associative algebra structures on $A \oplus M$ given in Proposition \ref{semi-prop}. Let $\pi_1, \pi_2 \in C^2 (A \oplus M, A \oplus M)$ denote the elements corresponding to the associative products on $A \oplus M$.

Note that any map $f \in C^n(A, M)$ can be lifts to a map $\widetilde{f} \in C^n (A \oplus M, A \oplus M)$ by
\begin{align*}
\widetilde{f} \big( (a_1, m_1), \ldots, (a_n, m_n)  \big) = \big( 0, f (a_1, \ldots, a_n ) \big),
\end{align*}
for $(a_i, m_i) \in A \oplus M$ and $i=1, \ldots, n$. Moreover, we have $f=0$ if and only if $\widetilde{f} = 0$. With all these notations, we have
\begin{align*}
\widetilde{(\delta_1 f )} = (-1)^{n-1}~[ \pi_1, \widetilde{f}]_\mathsf{G}  \qquad \text{ and } \qquad \widetilde{(\delta_2 f )} = (-1)^{n-1}~[ \pi_2, \widetilde{f}]_\mathsf{G},
\end{align*}
for $f \in C^n(A, M)$. We are now ready to prove the compatibility condition satisfied by $\delta_1$ and $\delta_2$. More precisely, we have the following.

\begin{pro}\label{delta-comp}
The coboundary operators $\delta_1$ and $\delta_2$ satisfy
\begin{align*}
\delta_1 \circ \delta_2 + \delta_2 \circ \delta_1 = 0.
\end{align*}
\end{pro}

\begin{proof}
For any $f \in C^n (A,M)$, we have
\begin{align*}
&\widetilde{( \delta_1 \circ \delta_2 + \delta_2 \circ \delta_1)(f)}  \\
&= (-1)^{n} ~  [\pi_1, \widetilde{\delta_2 f}]_\mathsf{G} ~+~ (-1)^{n} ~ [\pi_2, \widetilde{\delta_1 f}]_\mathsf{G} \\
&= - [\pi_1, [\pi_2, \widetilde{f}]_\mathsf{G} ]_\mathsf{G}  - [\pi_2, [\pi_1, \widetilde{f}]_\mathsf{G} ]_\mathsf{G}  \\
&= - [[\pi_1, \pi_2]_\mathsf{G}, \widetilde{f}]_\mathsf{G}  = 0 ~~~~ \qquad (\text{because} ~[\pi_1, \pi_2]_\mathsf{G} = 0).
\end{align*}
Therefore, it follows that  $  ( \delta_1 \circ \delta_2 + \delta_1 \circ \delta_1)(f) = 0$. Hence the result follows.
\end{proof}

The compatibility condition of the above proposition leads to cohomology associated with a compatible associative algebra with coefficients in a compatible bimodule. Let $A$ be a compatible associative algebra and $M$ be a compatible $A$-bimodule. We define the $n$-th cochain group $C^n_c (A, M)$, for $n \geq 0$, by
\begin{align*}
C^0_c (A, M) :=~& \{ m \in M ~|~ a \cdot_1 m - m \cdot_1 a = a \cdot_2 m - m \cdot_2 a, ~\forall a \in A \},\\
C^n_c (A, M) :=~& \underbrace{C^n (A, M) \oplus \cdots \oplus C^n (A, M)}_{n \text{ copies}}, ~ \text{ for } n \geq 1.
\end{align*}
Define a map $\delta_c : C^n_c (A, M) \rightarrow C^{n+1}_c (A, M)$, for $n \geq 0$, by
\begin{align}
\delta_c (m) (a) :=~&  a \cdot_1 m - m \cdot_1 a = a \cdot_2 m - m \cdot_2 a, \text{ for } m \in C^0_c (A, M) \text{ and } a \in A, \label{dc-1}\\
\delta_c (f_1, \ldots, f_n ) :=~& ( \delta_1 f_1, \ldots, \underbrace{\delta_1 f_i + \delta_2 f_{i-1}}_{i-\text{th place}}, \ldots, \delta_2 f_n),\label{dc-2}
\end{align}
for $(f_1, \ldots, f_n ) \in C^n_c (A, M)$. The map $\delta_c$ can be understood by the following diagram:

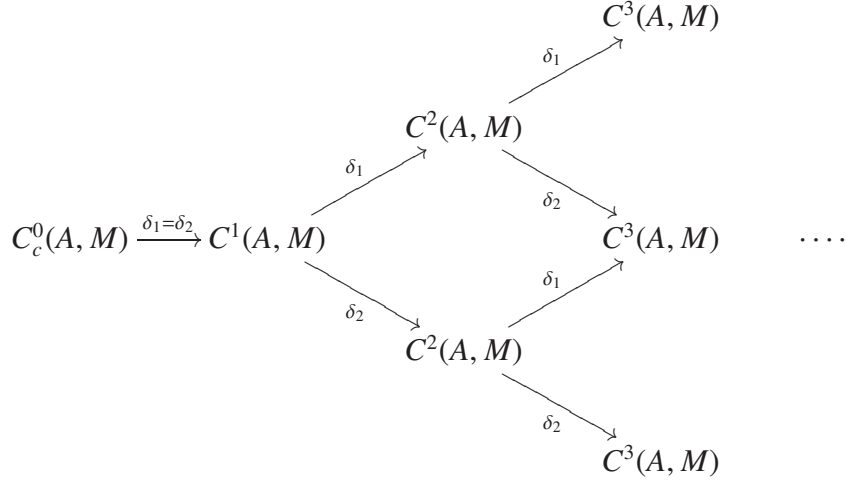
\begin{align*}
\xymatrix{
 & & &  C^3(A,M) & \\
 & & C^2(A,M)  \ar[ru]^{\delta_1} \ar[rd]_{\delta_2} & & \\
C^0_c (A,M) \ar[r]^{\delta_1 = \delta_2} & C^1(A,M) \ar[ru]^{\delta_1} \ar[rd]_{\delta_2} & & C^3(A,M) & \cdots \cdot\\
 & & C^2(A,M)  \ar[ru]^{\delta_1} \ar[rd]_{\delta_2} & & \\
 & & & C^3(A,M) &  \\
}
\end{align*}

Observe that, Proposition \ref{delta-comp} and the above diagrammatic presentation of $\delta_c$ shows that $(\delta_c)^2 =0$. However, we give a rigorous proof of the same fact.

\begin{pro}
The map $\delta_c$ is a coboundary operator, i.e., $(\delta_c)^2 = 0$.
\end{pro}

\begin{proof}
For $m \in C^0_c (A,M)$, we have
\begin{align*}
(\delta_c)^2 (m) = \delta_c ( \delta_c m ) =~& (\delta_1 \delta_c m~,\delta_2 \delta_c m ) \\
=~&  ( \delta_1 \delta_1 m~, \delta_2 \delta_2 m) = 0.
\end{align*}
Moreover, for any $(f_1, \ldots, f_n) \in C^n_c (A,M)$, $n \geq 1$, we have
\begin{align*}
(\delta_c)^2 (f_1, \ldots, f_n)
&= \delta_c \big(    \delta_1 f_1, \ldots,  \delta_1 f_i +  \delta_2 f_{i-1}, \ldots,  \delta_2 f_n \big) \\
&= \big(    \delta_1  \delta_1 f_1~,  \delta_2  \delta_1 f_1 +  \delta_1  \delta_2 f_1 +  \delta_1  \delta_1 f_2~, \ldots, \\
& \qquad \underbrace{   \delta_2  \delta_2 f_{i-2} +  \delta_2 \delta_1 f_{i-1} + \delta_1 \delta_2 f_{i-1} + \delta_1  \delta_1 f_i  }_{3 \leq i \leq n-1}~, \ldots, \\
& \qquad  \delta_2 \delta_2  f_{n-1} + \delta_2 \delta_1 f_n +  \delta_1 \delta_2 f_n ~,~  \delta_2  \delta_2 f_n \big) \\
&= 0 ~~~\quad (\text{from Proposition } \ref{delta-comp}).
\end{align*}
This proves that $(\delta_c)^2 = 0$.
\end{proof}

Thus, we have a cochain complex $\{ C^\ast_c (A, M), \delta_c \}$. Let $Z^n_c (A, M)$ denote the space of $n$-cocycles and $B^n_c (A, M)$ the space of $n$-coboundaries. Then we have $B^n_c (A, M) \subset Z^n_c (A, M)$, for $n \geq 0$. The corresponding quotient groups
\begin{align*}
H^n_c (A, M) := \frac{ Z^n_c (A, M) }{ B^n_c (A, M)}, \text{ for } n \geq 0
\end{align*}
are called the cohomology of the compatible associative algebra $A$ with coefficients in the compatible $A$-bimodule $M$.

\medskip

It follows from the above definition that
\begin{align*}
H^0_c (A, M) = \{ m \in M ~|~ a \cdot_1 m - m \cdot_1 a = a \cdot_2 m - m \cdot_2 a, ~ \forall a \in A \}.
\end{align*}

A linear map $D: A \rightarrow M$ is said to be a derivation on $A$ with values in the compatible $A$-bimodule $M$ if $D$ satisfies
\begin{align*}
D ( a \cdot_1 b ) = D(a) \cdot_1 b + a \cdot_1 D(b) \quad \text{ and } \quad D ( a \cdot_2 b ) = D(a) \cdot_2 b + a \cdot_2 D(b), \text{ for } a, b \in A.
\end{align*}
We denote the space of derivations by $\mathrm{Der}(A, M)$. A derivation $D$ is said to be an inner derivation if $D$ is of the form
$D (a) = a \cdot_1 m -  m \cdot_1 a = a \cdot_2 m -  m \cdot_2 a$, for some $m \in C^0_c (A, M).$ The space of inner derivations are denoted by $\mathrm{InnDer}(A, M).$ Then we have
\begin{align*}
H^1_c (A,M) = \frac{  \mathrm{Der}(A, M) }{ \mathrm{InnDer}(A, M) }.
\end{align*}

\subsection{Particular case: Cohomology with self coefficients}

Let $A = (A, \mu_1, \mu_2)$ be a compatible associative algebra. Then we have seen that $A$ is itself a compatible $A$-bimodule. If $\delta_1$ and $\delta_2$, respectively, denote the coboundary operator
for the Hochschild cohomology of $(A, \mu_1)$ and $(A, \mu_2)$ with self coefficients, then we have from (\ref{hoch-diff-brk}) that
\begin{align*}
\delta_1 f = (-1)^{n-1} [\mu_1, f]_G ~~~~ \text{ and } ~~~~ \delta_2 f = (-1)^{n-1} [\mu_2, f ]_G, ~ \text{ for } f \in C^n (A, A).
\end{align*}
Thus, it follows from (\ref{dc-2}) that the coboundary map $\delta_c : C^n_c (A, A) \rightarrow C^{n+1}_c(A, A)$ for the cohomology of the compatible associative algebra $A$ with self coefficients is given by
\begin{align}
\delta_c (f_1, \ldots, f_n ) =~& (-1)^{n-1} \big(  [\mu_1, f_1]_G, [\mu_1, f_2]_G + [\mu_2, f_1]_G, \ldots, [\mu_2, f_n]_G \big) \label{diff-gers}\\
=~& (-1)^{n-1} \llbracket (\mu_1, \mu_2), (f_1, \ldots, f_n ) \rrbracket, \nonumber
\end{align}
for $(f_1, \ldots, f_n ) \in C^n_c (A, A)$, $n \geq 1$. This shows that $\delta_c$ is same as the coboundary operator $d_{(\mu_1, \mu_2)}$  (defined in (\ref{d-mu})) up to a sign. Therefore, the corresponding cohomologies are isomorphic.

As a consequence, we get the following.

\begin{thm}\label{thm-gla-coho}
Let $A$ be a compatible associative algebra. Then the graded Lie bracket $\llbracket ~, ~ \rrbracket$ on $C^{\ast +1 }_c (A, A)$ induces a graded Lie bracket on the graded space $H^{\ast +1}_c (A, A)$ of cohomology groups.
\end{thm}

Let $(A, \mu_1, \mu_2)$ be a compatible associative algebra. Then we know that $(A, \mu_1 + \mu_2)$ is an associative algebra. Note that the cohomology of the compatible associative algebra $(A, \mu_1, \mu_2)$ with coefficients in itself is induced by the Maurer-Cartan element $(\mu_1, \mu_2) \in C^2_c (A, A)$ in the graded Lie algebra $(C^{\ast +1}_c (A, A), \llbracket ~, ~ \rrbracket)$. On the other hand, the Hochschild cohomology of the associative algebra $(A, \mu_1+ \mu_2)$ with coefficients in itself is induced by the Maurer-Cartan element $\mu_1 + \mu_2 \in C^2(A, A)$ in the graded Lie algebra $(C^{\ast +1} (A, A), [~,~]_G)$. Moreover, it follows from (\ref{the-phi}) that $\phi ((\mu_1, \mu_2)) = \mu_1 + \mu_2$. Hence, the graded Lie algebra map $\phi$ takes the Maurer-Cartan element $(\mu_1, \mu_2) \in C^2_c (A, A)$ to the Maurer-Cartan element $\mu_1 + \mu_2 \in C^2(A,A).$ Therefore, $\phi$ gives rise to a map between the cohomologies induced by Maurer-Cartan elements.

\begin{thm}
Let $(A, \mu_1, \mu_2)$ be a compatible associative algebra. Then the map (\ref{the-phi}) induces a morphism
\begin{align*}
\phi_* : H^\ast_c (A,A) \rightarrow H^\ast (A,A)
\end{align*}
from the cohomology of the compatible associative algebra $(A, \mu_1, \mu_2)$ with coefficients in itself to the Hochschild cohomology of the associative algebra $(A, \mu_1 + \mu_2)$ with coefficients in itself.
\end{thm}

\subsection{Relation with the cohomology of compatible Lie algebras}
Recently, the authors in \cite{comp-lie} introduced a cohomology theory for compatible Lie algebras. In this subsection, we show that our cohomology of compatible associative algebras is related to the cohomology of \cite{comp-lie} by the skew-symmetrization process. Let us first recall some results from the above-mentioned reference.

\begin{defi}
A compatible Lie algebra is a triple $(\mathfrak{g}, [~,~]_1, [~,~]_2)$ consists of a vector space $\mathfrak{g}$ together with two Lie brackets $[~,~]_1$ and $[~,~]_2$ satisfying the compatibility
\begin{align*}
[x,[y,z]_1]_2 + [y,[z,x]_1]_2 + [z,[x, y]_1]_2 + [x,[y,z]_2]_1 + [y,[z,x]_2]_1 + [z,[x, y]_2]_1 = 0, \text{ x, y, z } \in \mathfrak{g}.
\end{align*}
\end{defi}

\begin{defi}
Let $(\mathfrak{g}, [~,~]_1, [~,~]_2)$ be a compatible Lie algebra. A compatible $\mathfrak{g}$-representation is a triple $(V, \rho_1, \rho_2)$, where $(V, \rho_1)$ is a representation of the Lie algebra $(\mathfrak{g}, [~,~]_1)$ and $(V, \rho_2)$ is a representation of the Lie algebra $(\mathfrak{g}, [~,~]_2)$ satisfying additionally
\begin{align*}
\rho_2 ([x, y]_1) + \rho_1 ([x, y]_2) = \rho_1 (x) \rho_2 (y) - \rho_1 (y) \rho_2 (x) + \rho_2 (x) \rho_1 (y) - \rho_2 (y) \rho_1 (x),~\text{ for } x, y \in \mathfrak{g}.
\end{align*}
\end{defi}

Given a compatible Lie algebra $(\mathfrak{g}, [~,~]_1, [~,~]_2)$ and a compatible $\mathfrak{g}$-representation $(V, \rho_1, \rho_2)$, there is a cochain complex $\{ C^\ast_{cL} (\mathfrak{g}, V), \delta_{cL})$ defined as follows:
\begin{align*}
C^0_{cL} (\mathfrak{g}, V) :=~& \{ v \in V |~ \rho_1 (x) (v) = \rho_2 (x)(v), \forall x \in \mathfrak{g} \}, ~~ \text{ and } \\
C^n_{cL}  (\mathfrak{g}, V) :=~&  \underbrace{C^n_L  (\mathfrak{g}, V) \oplus \cdots \oplus  C^n_L (\mathfrak{g}, V)}_{n \text{ times}}, \text{ for } n \geq 1,
\end{align*}
where $C^n_L  (\mathfrak{g}, V) = \mathrm{Hom}(\wedge^n \mathfrak{g}, V).$ The coboundary operator $\delta_{cL} : C^n_{cL}  (\mathfrak{g}, V) \rightarrow C^{n+1}_{cL}  (\mathfrak{g}, V)$ is given by
\begin{align*}
\delta_{cL} (v) (x) := \rho_1 (x) (v) = \rho_2 (x) (v), ~ \text{ for } v \in C^0_{cL} (\mathfrak{g}, V), x \in \mathfrak{g},\\
\delta_{cL} (f_1, \ldots, f_n ) := ( \delta_{1L} f_1, \ldots, \underbrace{\delta_{1L} f_i + \delta_{2L} f_{i-1}}_{i\text{-th place}}, \ldots, \delta_{2L} f_n),
\end{align*}
for $(f_1, \ldots, f_n ) \in C^n_{cL} (\mathfrak{g}, V)$. Here $\delta_{1L}$ (resp. $\delta_{2L}$) is the coboundary operator for the Chevalley-Eilenberg cohomology of the Lie algebra $(\mathfrak{g}, [~,~]_1)$ with coefficients in $(V, \rho_1)$ ~~ (resp. of the Lie algebra $(\mathfrak{g}, [~,~]_2)$ with coefficients in $(V, \rho_2)$ ). The cohomology of the cochain complex $\{ C^\ast_{cL} (\mathfrak{g}, V), \delta_{cL} \}$ is called the cohomology of the compatible Lie algebra $(\mathfrak{g}, [~,~]_1, [~,~]_2)$ with coefficients in the compatible $\mathfrak{g}$-representation $(V, \rho_1, \rho_2)$, and they are denoted by $H^\ast_{cL} (\mathfrak{g}, V).$

\medskip

It is a well-known fact that the standard skew-symmetrization gives rise to a map from the Hochschild cochain complex of an associative algebra to the Chevalley-Eilenberg cohomology complex of the corresponding skew-symmetrized Lie algebra. This can be generalized to compatible algebras as well.

Let $(A, \mu_1, \mu_2)$ be a compatible associative algebra. Then it can be easily checked that the triple $(A, [~,~]_1, [~,~]_2)$
 is a compatible Lie algebra, where
 \begin{align*}
 [a, b]_1 := a \cdot_1 b - b \cdot_1 a ~~~~ \text{ and } ~~~~ [a, b]_2 := a \cdot_2 b - b \cdot_2 a, \text{ for } a, b \in A.
 \end{align*}
 We denote this compatible Lie algebra by $A_s$. Moreover, if $M$ is a compatible associative $A$-bimodule, then $M$ can be regarded as a compatible $A_s$-representation by
 \begin{align*}
 \rho_1 (a)(m) := a \cdot_1 m - m \cdot_1 a  ~~~~ \text{ and } ~~~~ \rho_2 (a)(m) := a \cdot_2 m - m \cdot_2 a, ~ \text{ for } a \in A_s, m \in M .
 \end{align*}
This compatible $A_s$-representation is denoted by $M_s$. With these notations, we have the following.

\begin{thm}
Let $A$ be a compatible associative algebra and $M$ be a compatible $A$-bimodule. Then the standard skew-symmetrization
\begin{align*}
\Phi_n : C^n_c (A,& M) \rightarrow C^n_{cL} (A_s, M_s),~ (f_1, \ldots, f_n) \mapsto (\overline{f_1}, \ldots, \overline{f_n}),~ \text{ for } n \geq 0, \\~~~ &\text{ where }
\overline{f_i} ( a_1, \ldots, a_n ) = \sum_{\sigma \in \mathbb{S}_n} (-1)^\sigma~ f_i (a_{\sigma (1)}, \ldots, a_{\sigma(n)}),~~~ i = 1, \ldots, n,
\end{align*}
gives rise a morphism of cochain complexes. Hence it induces a map $\Phi_\ast : H^\ast_c (A, M) \rightarrow H^\ast_{cL} (A_s, M_s)$.
\end{thm}

\subsection{Abelian extensions of compatible associative algebras}
In this subsection, we generalize the classical abelian extensions of associative algebras \cite{loday-book} to the context of compatible associative algebras. We show that equivalence classes of abelian extensions of a compatible associative algebra are characterized by the second cohomology group of the compatible associative algebra.

Let $(A, \mu_1, \mu_2)$ be a compatible associative algebra and $M$ be any vector space. Note that $M$ can also be considered as a compatible associative algebra with trivial associative products.

\begin{defi}
An abelian extension of a compatible associative algebra $A$ be a vector space $M$ is an exact sequence
\begin{align}\label{abel-exact-seq}
\xymatrix{
0 \ar[r] &  (M, 0, 0) \ar[r]^{i} & (B, \mu_1^B, \mu_2^B) \ar[r]^{j} & (A, \mu_1, \mu_2) \ar[r] & 0
}
\end{align}
of compatible associative algebras.
\end{defi}

It is important to note that an abelian extension is the whole exact sequence (including the structure maps $i$ and $j$), not just the compatible associative algebra $(B, \mu_1^B, \mu_2^B)$.

Let $s: A \rightarrow B$ be any map satisfying $j \circ s = \mathrm{id}_A$. Such a map always exist. In this case, $s$ is called a section of the map $j$. A section $s$ induces a compatible $A$-bimodule structure on $M$ given by
\begin{align*}
\begin{cases}
a \cdot_1 m = \mu_1^B ( s(a), i(m)), \\
 m \cdot_1 a = \mu_1^B (i(m), s(a)),
 \end{cases} \qquad \quad
 \begin{cases}
a \cdot_2 m = \mu_2^B ( s(a), i(m)), \\
 m \cdot_2 a = \mu_2^B (i(m), s(a)),
 \end{cases}
\end{align*}
for $a \in A$ and $m \in M$. One can easily check that this compatible $A$-bimodule structure on $M$ is independent of the choice of $s$.

\begin{defi}
Two abelian extensions (two horizontal rows in the below diagram) of a compatible associative algebra $A$ by a vector space $M$ are said to be equivalent if there is a compatible associative algebra morphism $\phi : B \rightarrow B'$ making the following diagram commutative
\[
\xymatrix{
0 \ar[r] &  (M, 0, 0)  \ar[r]^{i} \ar@{=}[d] & (B, \mu_1^B, \mu_2^B) \ar[d]^{\phi} \ar[r]^{j} & (A, \mu_1, \mu_2) \ar[r]  \ar@{=}[d]  & 0 \\
0 \ar[r] &  (M, 0, 0) \ar[r]_{i'} & (B', \mu_1^{B'}, \mu_2^{B'}) \ar[r]_{j'} & (A, \mu_1, \mu_2) \ar[r] & 0.
}
\]
\end{defi}

Let $A$ be a compatible associative algebra and $M$ be a given compatible $A$-bimodule. We denote by Ext$(A,M)$ the set of equivalence classes of abelian extensions of $A$ by the vector space $M$ so that the induced compatible $A$-bimodule structure on $M$ is the prescribed one.

Then we have the following which generalizes the classical result \cite{loday-book} about abelian extensions.

\begin{thm}
Let $A$ be a compatible associative algebra and $M$ be a compatible $A$-bimodule. Then there is a bijection between $\mathrm{Ext}(A,M)$ and the second cohomology group $H^2_c (A, M).$
\end{thm}

\begin{proof}
Let $(f_1, f_2) \in Z^2_c (A, M)$ be a $2$-cocycle. Then it is easy to see that the direct sum $A \oplus M$ carries a compatible associative algebra structure given by
\begin{align*}
\mu_1^B ((a,m), (b, n)) :=~& ( a \cdot_1 b ,~ a \cdot_1 n + m \cdot_1 b + f_1 (a, b)), \\
\mu_2^B ((a,m), (b, n)) :=~& ( a \cdot_2 b ,~ a \cdot_2 n + m \cdot_2 b + f_2 (a, b)),
\end{align*}
for $(a, m), (b, n) \in A \oplus M$. Moreover, the exact sequence
\begin{align*}
\xymatrix{
0 \ar[r] &  (M, 0, 0) \ar[r]^{i} & (B, \mu_1^B, \mu_2^B) \ar[r]^{j} & (A, \mu_1, \mu_2) \ar[r] & 0
}
\end{align*}
defines an abelian extension, where $i (m) = (0, m)$ and $j (a, m) = a$. Suppose $(f_1', f_2') \in Z^2_c (A, M)$ is another $2$-cocycle cohomologous to $(f_1, f_2)$, and say,
\begin{align*}
(f_1, f_2) - (f_1', f_2') = \delta_c (g), ~\text{ for some } g \in C^1_c (A, M).
\end{align*}
Let $B' = A \oplus M$ be the abelian extension corresponding to the $2$-cocycle $(f_1', f_2')$. Then the two abelian extensions are equivalent and the equivalence is given by the compatible associative algebra map $\phi : B \rightarrow B'$, $(a,m) \mapsto (a, m + g (a))$. In other words, we obtain a well-defined map $H^2_c (A, M) \rightarrow \mathrm{Ext}(A,M).$

Conversely, let (\ref{abel-exact-seq}) be an abelian extension and $s: A \rightarrow B$ be any section of $j$. Then we may consider $B = A \oplus M$ and the maps $i, j, s$ are the obvious ones. Since $j$ is a compatible associative algebra map, we have
\begin{align*}
j \circ \mu_1^B ((a, 0), (b, 0)) = \mu_1 (a, b) ~~~ \text{ and } ~~~
j \circ \mu_2^B ((a, 0), (b, 0)) = \mu_2 (a, b), ~ \text{ for } a, b \in A.
\end{align*}
Hence it follows that
\begin{align*}
\mu_1^B ((a, 0), (b,0)) = (\mu_1 (a, b), f_1 (a, b)) ~~ \text{ and } ~~ \mu_2^B ((a, 0), (b,0)) = (\mu_2 (a, b), f_2 (a, b)),
\end{align*}
for some $f_1, f_2 : A^{\otimes 2} \rightarrow M$. Since $\mu_1^B, \mu_2^B$ defines a compatible associative structure on $B$, it follows that the pair $(f_1, f_2) \in C^2_c (A, M)$ is a $2$-cocycle in the cohomology of the compatible associative algebra $A$ with coefficients in the compatible $A$-bimodule $M$. It is left to the reader to verify that equivalent abelian extensions induce cohomologous $2$-cocycles (see \cite{loday-book} for the classical associative case). This shows that there is a well-defined map $\mathrm{Ext}(A,M) \rightarrow H^2_c (A, M)$.

Finally, the maps $H^2_c (A, M) \rightarrow \mathrm{Ext}(A,M)$ and $\mathrm{Ext}(A,M) \rightarrow H^2_c (A, M)$ constructed above are inverses to each other. Hence the proof.
\end{proof}

\section{Deformations of compatible associative algebras}\label{section:deformation}
In this section, we study various aspects of deformations of compatible associative algebras following the classical deformation theory of Gerstenhaber \cite{gers}.

\subsection{Linear deformations and Nijenhuis operators} In this subsection, we consider linear deformations of a compatible associative algebra $A$ and introduce Nijenhuis operators on $A$ that induce trivial linear deformations. We also introduce infinitesimal deformations of $A$ and show that equivalence classes of infinitesimal deformations are in one-to-one correspondence with the second cohomology group $H^2_c (A, A).$

Let $A= (A, \mu_1, \mu_2)$ be a compatible associative algebra.
\begin{defi}
A linear deformation of $A$ consists of two linear sums of the form
\begin{align*}
\mu_1^t = \mu_1 + t \omega_1 ~~~ \text{ and } ~~~ \mu_2^t = \mu_2 + t \omega_2, ~\text{ for some } \omega_1, \omega_2 \in C^2(A,A)
\end{align*}
which makes $(A, \mu_1^t, \mu_2^t)$ into a compatible associative algebra, for all values of $t$.
\end{defi}

In this case, we say that the pair $(\omega_1, \omega_2)$ generates a linear deformation of $A$.

\begin{ex}
The pair $(\mu_1, \mu_2)$ generates a linear deformation of $A$. It is called the `scaling'.
\end{ex}

Let $(\omega_1, \omega_2)$ generates a linear deformation of the compatible associative algebra $A$. It follows that $\mu_1^t = \mu_1 + t \omega_1$ and $\mu_2^t = \mu_2 + t \omega_2$ satisfy the following relations
\begin{align*}
[\mu_1^t, \mu_1^t]_\mathsf{G} = 0, \qquad  [\mu_2^t, \mu_2^t]_\mathsf{G} = 0 ~~~~ \text{ and } ~~~~ [\mu_1^t, \mu_2^t]_\mathsf{G} = 0.
\end{align*}
These relations are equivalent to the followings:
\begin{align}
[\mu_1, \omega_1]_\mathsf{G} = 0, \qquad [\mu_2, \omega_2]_\mathsf{G} = 0, \qquad [\mu_1, \omega_2]_\mathsf{G} + [\mu_2 , \omega_1]_\mathsf{G} = 0, \label{lin-f}\\
[\omega_1, \omega_1]_\mathsf{G} = 0, \qquad  [\omega_2, \omega_2]_\mathsf{G} = 0 \qquad \text{ and } \qquad [\omega_1, \omega_2 ]_\mathsf{G} = 0. \label{lin-s}
\end{align}
The three identities in (\ref{lin-f}) implies that
\begin{align*}
\delta_c (\omega_1, \omega_2) = 0,
\end{align*}
where $\delta_c$ is the coboundary operator for the cohomology of the compatible associative algebra $A$ with coefficients in itself. In other words, $(\omega_1, \omega_2) \in Z^2_c (A, A)$ is a $2$-cocycle.

On the other hand, the three conditions of (\ref{lin-s}) imply that the triple $(A, \omega_1, \omega_2)$ is a compatible associative algebra.

\begin{defi}
Let $(\omega_1, \omega_2)$ generates a linear deformation $(\mu_1^t, \mu_2^t)$ of the compatible associative algebra $A$, and $(\omega_1', \omega_2')$ generates another linear deformation $(\mu_1^{'t}, \mu_2^{'t})$ of $A$. They are said to be equivalent if there is a linear map $N : A \rightarrow A$ such that
\begin{align*}
\mathrm{id} + t N : (A, \mu_1^{t}, \mu_2^{t}) \rightarrow (A, \mu_1^{'t}, \mu_2^{'t})
\end{align*}
is a morphism of compatible associative algebras.
\end{defi}

The condition in the above definition implies that
\begin{align*}
(\mathrm{id} + tN)\circ \mu_i^{t} (a, b) = \mu_i^{'t} ( a + t Na, b + tNb),~\text{ for } i = 1,2 \text{ and } a, b \in A.
\end{align*}
If we write down explicitly, we get that
\begin{align}
\omega_i ( a, b ) - \omega_i' (a, b) =~& a \cdot_i N(b) + N(a) \cdot_i b - N ( a \cdot_i b), \label{lin-eq-1}\\
N\omega_i (a, b) =~& \omega_i' ( a, Nb) + \omega_i' (Na, b) + N(a) \cdot_i N(b), \label{lin-eq-2}\\
\omega_i' (Na, Nb) =~& 0, \label{lin-eq-3}
\end{align}
for $i=1,2$ and $a, b \in A$. From the two identities (for $i=1,2$) of (\ref{lin-eq-1}), we simply get that
\begin{align}\label{d-f}
(\omega_1, \omega_2) - (\omega_1', \omega_2') = \delta_c (N),
\end{align}
where we consider $N$ as an element in $C^1_c (A, A).$ As a summary of the previous discussions, we get the following.

\begin{thm}\label{thm-lin-def}
There is a map from the set of equivalence classes of linear deformations of a compatible associative algebra $A$ to the second cohomology group $H^2_c (A,A).$
\end{thm}

\medskip

Let us now discuss trivial linear deformations of a compatible associative algebra $A$.

\begin{defi}
A linear deformation generated by $(\omega_1, \omega_2)$ is said to be trivial if the deformation is equivalent to the undeformed one (i.e., generated by $(\omega_1', \omega_2') = (0,0))$.
\end{defi}

Thus, it follows from (\ref{lin-eq-1})-(\ref{lin-eq-3}) that a linear deformation generated by $(\omega_1, \omega_2)$ is trivial if there is a linear map $N : A \rightarrow A$ satisfying
\begin{align}
\omega_i (a, b) =~& a \cdot_i N(b) + N(a) \cdot_i b - N ( a \cdot_i b), \label{lin-nij-1}\\
N \omega_i (a, b) =~& N(a) \cdot_i N(b), ~ \text{ for } i=1,2 \text{ and } a, b \in A. \label{lin-nij-2}
\end{align}

These two identities motivate us to introduce the following definition.

\begin{defi}
Let $A = (A, \mu_1, \mu_2)$ be a compatible associative algebra. A linear map $N: A \rightarrow A$ is said to be a Nijenhuis operator on $A$ if $N$ is a Nijenhuis operator for both the associative products $\mu_1$ and $\mu_2$. In other words,
\begin{align*}
N(a) \cdot_i N(b) = N \big( a \cdot_i N(b) + N(a) \cdot_i b - N ( a \cdot_i b) \big),~\text{ for } i=1,2 \text{ and } a, b \in A.
\end{align*}
\end{defi}

It follows from the identities (\ref{lin-nij-1}) and (\ref{lin-nij-2}) that a trivial linear deformation of a compatible associative algebra $A$ induces a Nijenhuis operator on $A$. The converse is also true which is stated in the next proposition.

\begin{pro}
Let $N: A \rightarrow A$ be a Nijenhuis operator on a compatible associative algebra $A$. Then $N$ induces a trivial linear deformation of $A$ generated by $(\omega_1, \omega_2)$, where
\begin{align*}
\omega_i (a, b) := a \cdot_i N(b) + N(a) \cdot_i b - N ( a \cdot_i b),~\text{ for } i=1,2 \text{ and } a, b \in A.
\end{align*}
\end{pro}

\begin{proof}

First observe that $(\omega_1, \omega_2) = \delta_c (N)$. Therefore, we have
\begin{align*}
\delta_c (\omega_1, \omega_2) = 0.
\end{align*}
This is equivalent to the conditions in (\ref{lin-f}). Moreover, since $N$ is a Nijenhuis operator on the compatible associative algebra $(A, \mu_1, \mu_2)$, it follows that $(A, \omega_1 = (\mu_1)_N, \omega_2 = (\mu_2)_N)$ is a compatible associative algebra. In other words, the conditions in (\ref{lin-s}) hold. This implies that $(\omega_1, \omega_2)$ generates a linear deformation of the compatible associative algebra $A$.

Finally, this linear deformation obviously satisfies the conditions in (\ref{lin-nij-1}) and (\ref{lin-nij-2}) which implies that the linear deformation is trivial.
\end{proof}

\medskip

In Theorem \ref{thm-lin-def}, we have seen that there is a map
\begin{align}\label{lin-sim}
(\text{linear deformations of } A)/ \sim  ~ \rightarrow ~ H^2_c(A,A)
\end{align}
from the set of equivalence classes of linear deformations of a compatible associative algebra $A$ to the second cohomology group $H^2_c (A, A).$ To generalize the map (\ref{lin-sim}) into certain isomorphism, we introduce the notion of infinitesimal deformations of a compatible associative algebra $A$.

\begin{defi}
An infinitesimal deformation of a compatible associative algebra $A$ is a linear deformation of $A$ over the base $\mathbb{K}[[t]]/ (t^2)$, the ring of dual numbers.
\end{defi}

One can also define equivalences between two infinitesimal deformations of $A$. Any $2$-cocycle $(\omega_1, \omega_2) \in Z^2_c (A,A)$ induces an infinitesimal deformation $(\mu_1^t = \mu_1 + t \omega_1, \mu_2^t = \mu_2 + t \omega_2)$ of $A$. Moreover, cohomologous $2$-cocycles induce equivalent infinitesimal deformations. More precisely, let $(\omega_1', \omega_2') \in Z^2_c (A,A)$ be another $2$-cocycle and cohomologous to $(\omega_1, \omega_2)$, say
\begin{align*}
(\omega_1, \omega_2) - (\omega_1', \omega_2') = \delta_c h, ~\text{ for some } h \in C^1_c (A, A).
\end{align*}
Then the infinitesimal deformations $(\mu_1^t = \mu_1 + t \omega_1, \mu_2^t = \mu_2 + t \omega_2)$ and $(\mu_1^{'t} = \mu_1 + t \omega'_1, \mu_2^{'t} = \mu_2 + t \omega'_2)$ are equivalent and an equivalence is given by $\mathrm{id} + t h : (A, \mu_1^t, \mu_2^t) \rightarrow (A, \mu_1^{'t}, \mu_2^{'t})$. Hence, summarizing this fact with Theorem \ref{thm-lin-def}, we obtain the following.

\begin{thm}
Let $A$ be a compatible associative algebra. Then there is a one-to-one correspondence between the set of equivalence classes of infinitesimal deformations of $A$ and the second cohomology group $H^2_c (A,A).$
\end{thm}

\subsection{Formal deformations and rigidity}

In this subsection, we consider formal deformations of a compatible associative algebra $A$. We show that the vanishing of the second cohomology group $H^2_c (A,A)$ implies that $A$ is rigid.

Let $A = (A, \mu_1, \mu_2)$ be a compatible associative algebra. Consider the space $A[[t]]$ of formal power series in $t$ with coefficients from $A$. Then $A[[t]]$ is a $\mathbb{K}[[t]]$-module.

\begin{defi}
A formal deformation of the compatible associative algebra $A$ consists of two formal power series
\begin{align*}
\mu_{1,t} :=~& \mu_{1,0} + t \mu_{1,1} + t^2 \mu_{1,2} + \cdots,\\
\mu_{2,t} :=~& \mu_{2,0} + t \mu_{2,1} + t^2 \mu_{2,2} + \cdots,
\end{align*}
where $\mu_{1, i}$'s and $\mu_{1, i}$'s are bilinear maps on $A$ with $\mu_{1,0} = \mu_1$ and $\mu_{2,0} = \mu_2$ such that $(A[[t]], \mu_{1,t}, \mu_{2,t})$ is a compatible associative algebra over $\mathbb{K}[[t]].$
\end{defi}

Thus, $(\mu_{1,t}, \mu_{2,t})$ is a formal deformation of $A$ if and only if
\begin{align*}
[\mu_{1,t}, \mu_{1,t}]_G = 0, \qquad  [\mu_{2,t}, \mu_{2,t}]_G = 0 ~~ \text{ and } ~~ [\mu_{1,t}, \mu_{2,t}]_G = 0.
\end{align*}
They are equivalent to the following identities
\begin{align}\label{def-eqn}
\sum_{i+j = k} [\mu_{1,i}, \mu_{1,j}]_G = 0, \qquad \sum_{i+j = k} [\mu_{2,i}, \mu_{2,j}]_G = 0 ~~ \text{ and } ~~ \sum_{i+j = k} [\mu_{1,i}, \mu_{2,j}]_G = 0,
\end{align}
for $k=0,1,2, \ldots ~.$ The equations (\ref{def-eqn}) are called deformation equations.

Observe that the identities are hold for $k=0$ as $(\mu_1, \mu_2)$ defines a compatible associative algebra structure on $A$. However, for $k=1$, we get
\begin{align*}
[\mu_1, \mu_{1,1}]_G = 0, \qquad [\mu_2, \mu_{2,1}]_G = 0 ~~ \text{ and } ~~ [\mu_1, \mu_{2,1}]_G + [\mu_2 , \mu_{1,1}]_G =0.
\end{align*}
Hence it follows from (\ref{diff-gers}) that $(\mu_{1,1}, \mu_{2,1})$ is a $2$-cocycle in the cohomology of $A$ with coefficients in itself. It is called the `infinitesimal' of the deformation $(\mu_{1,t}, \mu_{2,t})$.

\begin{rmk}\label{n-co-def}
Let $(\mu_{1,t}, \mu_{2,t})$ be a formal deformation of the form $\mu_{1,t} = \mu_1 + \sum_{i \geq p} t^i \mu_{1,i}$ and $\mu_{2,t} = \mu_2 + \sum_{i \geq p} t^i \mu_{2,i}$. Then one can show that the pair $(\mu_{1,p}, \mu_{2,p})$ is a $2$-cocycle in the cohomology of $A$ with coefficients in itself.
\end{rmk}

\begin{defi}
Two formal deformations $(\mu_{1,t}, \mu_{2,t})$ and $(\mu_{1,t}', \mu_{2,t}')$ of a compatible associative algebra $A$ are said to be equivalent if there is a formal power series
\begin{align*}
\phi_t := \phi_0 + t \phi_1 + t^2 \phi_2 + \cdots   ~~~~ (\text{with } \phi_0 = \mathrm{id}_A),
\end{align*}
where $\phi_i$'s are linear maps on $A$ such that the $\mathbb{K}[[t]]$-linear map $\phi_t : (A[[t]], \mu_{1,t}, \mu_{2,t}) \rightarrow (A[[t]], \mu_{1,t}', \mu_{2,t}')$ is a morphism of compatible associative algebras.
\end{defi}

By using the fact that $\phi_t$ is a compatible associative algebra morphism, we get that (similar to (\ref{d-f}))
\begin{align*}
(\mu_{1,1}, \mu_{2,1}) - (\mu_{1,1}', \mu_{2,1}') = \delta_c (\phi_1).
\end{align*}
This shows that infinitesimals corresponding to equivalent deformations are cohomologous, hence, corresponds to the same cohomology class in $H^2_c (A, A).$

\medskip

Next, we define the notion of rigidity of a compatible associative algebra.

\begin{defi}
A compatible associative algebra $A$ is said to be rigid if any formal deformation of $A$ is equivalent to the undeformed one.
\end{defi}

\begin{pro}
Let $(\mu_{1,t}, \mu_{2,t})$ be a formal deformation of a compatible associative algebra $A$. Then it is equivalent to a deformation $(\mu_{1,t}' = \mu_1 + \sum_{i \geq p} t^i \mu_{1,i}',~ \mu_{2,t}' = \mu_2 + \sum_{i \geq p} t^i \mu_{2,i}')$, where the first non-vanishing pair $(\mu_{1,p}', \mu_{2,p}')$ is a $2$-cocycle but not a $2$-coboundary.
\end{pro}

\begin{proof}
Suppose $\mu_{1,t}$ and $\mu_{2,t}$ are of the form
\begin{align*}
\mu_{1,t} = \mu_{1,0} + t^n \mu_{1,n} + t^{n+1} \mu_{1,n+1} + \cdots ~~ \text{ and } ~~ \mu_{2,t} = \mu_{2,0} + t^n \mu_{2,n} + t^{n+1} \mu_{2,n+1} + \cdots.
\end{align*}
Then we know from Remark \ref{n-co-def} that $(\mu_{1,n}, \mu_{2,n})$ is a $2$-cocycle in the cohomology of $A$ with coefficients in itself. If $(\mu_{1,n}, \mu_{2,n})$ is not a $2$-coboundary, we are done. However, if $(\mu_{1,n}, \mu_{2,n})$ is a $2$-coboundary, say $- \delta_c (\phi_n)$, then we set $\phi_t = \mathrm{id}_A + t^n \phi_n$. Define
\begin{align*}
\mu_{1,t}' = \phi_t^{-1} \circ \mu_{1,t} \circ ( \phi_t \otimes \phi_t ) ~~ \text{ and } ~~ \mu_{2,t}' = \phi_t^{-1} \circ \mu_{2,t} \circ ( \phi_t \otimes \phi_t ).
\end{align*}
Then $(\mu_{1,t}', \mu_{2,t}')$ is a deformation equivalent to $(\mu_{1,t}, \mu_{2,t})$, and further the coefficients of $t, t^2, \ldots, t^n$ in $\phi_{1,t}'$ and $\phi_{2,t}'$ are zero. By repeating this process, we get the required deformation.
\end{proof}

As a consequence, we get the following sufficient condition for rigidity.

\begin{thm}
Let $A$ be a compatible associative algebra. If $H^2_c(A, A) = 0$ then $A$ is rigid.
\end{thm}

\subsection{Finite order deformations and their extensions}

In this subsection, we consider finite order deformations of a compatible associative algebra $A$ and their extensions to deformations of the next order. We show that the corresponding obstruction class for such extension lies in the third cohomology group of $A$.

Let $(A, \mu_1, \mu_2)$ be a compatible associative algebra and let $N \in \mathbb{N}$ be a fixed natural number. Consider the space $A[[t]]/(t^{N+1})$ which is a module over the ring $\mathbb{K}[[t]]/(t^{N+1})$.

\begin{defi}
An order $N$ deformation of the compatible associative algebra $A$ consists of two polynomials of the form
\begin{align*}
\mu_{1,t}^N :=~& \mu_{1,0} + t \mu_{1,1} + t^2 \mu_{1,2} + \cdots + t^N \mu_{1,N},\\
\mu_{2,t}^N :=~& \mu_{2,0} + t \mu_{2,1} + t^2 \mu_{2,2} + \cdots + t^N \mu_{2,N},
\end{align*}
where $\mu_{i,j}$'s are bilinear maps on $A$ with $\mu_{1,0} = \mu_1$ and $\mu_{2,0} = \mu_2$ such that $(A[[t]]/ (t^{N+1}), \mu_{1,t}^N , \mu_{2,t}^N)$ is a compatible associative algebra over $\mathbb{K}[[t]]/(t^{N+1})$.
\end{defi}

\begin{ex}
An infinitesimal deformation of $A$ is an order $1$ deformation of $A$.
\end{ex}

Let $(\mu_{1,t}^N, \mu_{2,t}^N)$ be a deformation of order $N$. Then the following set of relations are hold:
\begin{align*}
\sum_{i+j = k} [\mu_{1,i}, \mu_{1,j}]_G = 0, ~~~~  \sum_{i+j = k} [\mu_{2,i}, \mu_{2,j}]_G = 0 ~~~~ \text{ and } ~~~~ \sum_{i+j = k} [\mu_{1,i}, \mu_{2,j}]_G = 0,
\end{align*}
for $k=0,1, \ldots, N$. These three sets of relations can be compactly written as
\begin{align}\label{finite-def}
\llbracket (\mu_1, \mu_2), ( \mu_{1, k}, \mu_{2,k}) \rrbracket = - \frac{1}{2} \sum_{i+j = k;~ i, j \geq 1} \llbracket (\mu_{1,i}, \mu_{2,i}), (\mu_{1,j}, \mu_{2,j}) \rrbracket, \text{ for } k=1, \ldots , N.
\end{align}

\begin{defi}
An order $N$ deformation $(\mu_{1,t}^N = \sum_{i=0}^N t^i \mu_{1,i}, \mu_{2,t}^N = \sum_{i=0}^N t^i \mu_{2,i})$ is said to be extensible if there exist two bilinear maps on $A$ (say, $\mu_{1, N+1}$ and $\mu_{2, N+1})$ such that
\begin{align*}
\big( \mu_{1,t}^{N+1} := \mu_{1,t}^N + t^{N+1} \mu_{1, N+1},~ \mu_{2,t}^{N+1} := \mu_{2,t}^N + t^{N+1} \mu_{2, N+1} \big)
\end{align*}
is a deformation of order $N+1$.
\end{defi}

Note that, for $(\mu_{1,t}^{N+1}, \mu_{2,t}^{N+1})$ to be a deformation of order $N+1$, we must have the relations (\ref{finite-def}) and one more equation to be satisfied, namely,
\begin{align}\label{finite-def-n1}
\llbracket (\mu_1, \mu_2), ( \mu_{1, N+1}, \mu_{2,N+1}) \rrbracket = - \frac{1}{2} \sum_{i+j = N+1;~ i, j \geq 1} \llbracket (\mu_{1,i}, \mu_{2,i}), (\mu_{1,j}, \mu_{2,j}) \rrbracket.
\end{align}
Observe that the right hand side of (\ref{finite-def-n1}) is a $3$-cochain on the cohomology complex of the compatible associative algebra $A$. Moreover, it does not contain $\mu_{1, N+1}$ and $\mu_{2, N+1}$. Hence it depends only on the order $N+1$ deformation $(\mu_{1,t}^N, \mu_{2,t}^N)$. It is called the obstruction cochain to extend the deformation $(\mu_{1,t}^N, \mu_{2,t}^N)$, denoted by $Ob_{(\mu_{1,t}^N, \mu_{2,t}^N)}$.

\begin{pro}
The obstruction cochain $Ob_{(\mu_{1,t}^N, \mu_{2,t}^N)}$ is a $3$-cocycle, i.e., $\delta_c \big(   Ob_{(\mu_{1,t}^N, \mu_{2,t}^N)} \big) = 0$.
\end{pro}

\begin{proof}
We have
\begin{align*}
&\delta_c \big(  Ob_{(\mu_{1,t}^N, \mu_{2,t}^N)} \big) \\
&= - \frac{1}{2} \sum_{i+j = N+1;~ i, j \geq 1}  \llbracket (\mu_1, \mu_2), \llbracket   (\mu_{1,i}, \mu_{2,i}), (\mu_{1,j}, \mu_{2,j}) \rrbracket  \rrbracket \\
&= - \frac{1}{2} \sum_{i+j = N+1;~ i, j \geq 1}  \big(      \llbracket \llbracket  (\mu_{1}, \mu_{2}), (\mu_{1,i}, \mu_{2,i}) \rrbracket, (\mu_{1,j}, \mu_{2,j}) \rrbracket  ~-~ \llbracket (\mu_{1,i}, \mu_{2,i}), \llbracket (\mu_{1}, \mu_{2}), (\mu_{1,j}, \mu_{2,j}) \rrbracket \rrbracket \big) \\
&= \frac{1}{4} \sum_{i_1 + i_2 + j = N+1;~ i_1, i_2, j \geq 1} \llbracket \llbracket (\mu_{1,i_1}, \mu_{2,i_1}), (\mu_{1,i_2}, \mu_{2,i_2}) \rrbracket, (\mu_{1,j}, \mu_{2,j}) \rrbracket  \\
& \qquad - \frac{1}{4} \sum_{i+j_1 + j_2 = N+1;~ i, j_1, j_2 \geq 1} \llbracket (\mu_{1,i}, \mu_{2,i}), \llbracket (\mu_{1,j_1}, \mu_{2,j_1}), (\mu_{1,j_2}, \mu_{2,j_2}) \rrbracket \rrbracket    \qquad (\text{form } (\ref{finite-def})) \\
&= \frac{1}{2} \sum_{i+j+k = N+1;~ i, j, k \geq 1} \llbracket \llbracket (\mu_{1,i}, \mu_{2,i}), (\mu_{1,j}, \mu_{2,j}) \rrbracket , (\mu_{1,k}, \mu_{2,k}) \rrbracket  = 0.
\end{align*}
Hence the proof.
\end{proof}

Thus, it follows that the obstruction cochain induces a third cohomology class $[Ob_{(\mu_{1,t}^N, \mu_{2,t}^N)}] \in H^3_c (A, A)$. It is called the obstruction class. Finally, as a consequence of (\ref{finite-def-n1}), we get the following.

\begin{thm}
An order $N$ deformation $(\mu_{1,t}^N, \mu_{2,t}^N)$ of a compatible associative algebra $A$ is extensible if and only if the corresponding obstruction class $[Ob_{(\mu_{1,t}^N, \mu_{2,t}^N)}] \in H^3_c (A, A)$  is trivial.
\end{thm}

As a corollary, we obtain these interesting results.

\begin{cor}
If $H^3_c (A,A) =0$ then any finite order deformation of $A$ is extensible.
\end{cor}

\begin{cor}
If $H^3_c(A,A)= 0 $ then any $2$-cocycle $(\omega_1, \omega_2) \in Z^2_c (A, A)$ is the infinitesimal of some formal deformations of $A$.
\end{cor}

\section{Homology of compatible associative algebras}\label{section:homology}

In this section, we introduce a notion of compatible presimplicial vector space and associate a chain complex to any compatible presimplicial vector space. As an application, we introduce the homology of compatible associative algebra $A$ with coefficients in a compatible $A$-bimodule.

\subsection{Compatible presimplicial vector spaces}

\begin{defi} \cite{loday-book}
A presimplicial vector space $(C, \Delta)$ consists of a collection $C = \{ C_n \}_{n \geq 0}$ of vector spaces, together with a collection of linear maps (called face maps)
\begin{align*}
\Delta = \{ \partial_i : C_n \rightarrow C_{n-1} ~|~ i=0,1, \ldots, n \}_{n \geq 1}
\end{align*}
satisfying
\begin{align*}
\partial_i \partial_j = \partial_{j-1} \partial_i,~\text{ for } 0 \leq i < j \leq n.
\end{align*}
\end{defi}

Given a presimplicial vector space $(C, \Delta)$, one can define a map $\partial : C_n \rightarrow C_{n-1}$, for $n \geq 1$, by $\partial = \sum_{i=0}^n (-1)^i \partial_i$.
Then it is easy to see that $\partial$ is a differential, i.e., $\partial^2 = 0$. In other words, $(C, \partial)$ is a chain complex.

\begin{defi}
A compatible presimplicial vector space is a triple $(C, \Delta, \Delta')$ in which $(C, \Delta)$ and $(C, \Delta')$ are presimplicial vector spaces satisfying additionally
\begin{align}\label{presim-id}
\partial_i \partial_j' + \partial_i' \partial_j = \partial_{j-1} \partial_i' + \partial_{j-1}' \partial_i, ~ \text{ for } 0 \leq i < j \leq n.
\end{align}
\end{defi}

Let $(C, \Delta, \Delta')$ be a compatible presimplicial vector space. Then we have the following.

\begin{pro}\label{comp-pre-pro}
The induced differentials $\partial$ and $\partial'$ satisfy ~$\partial \partial' + \partial' \partial = 0$.
\end{pro}

\begin{proof}
First observe that
\begin{align*}
\partial \partial' + \partial' \partial = \sum_i \sum_j (-1)^{i+j}~ \partial_i \partial_j' + \sum_i \sum_j (-1)^{i+j}~ \partial_i' \partial_j.
\end{align*}
We now split both of these sums into two parts: namely for $i <j$ and $i \geq j$. Thus,
\begin{align*}
\partial \partial' + \partial' \partial =~& \sum_{i <j } (-1)^{i+j} ~ ( \partial_i \partial_j' + \partial_i' \partial_j ) ~+~ \sum_{i \geq j} (-1)^{i+j} ~ ( \partial_i \partial_j' + \partial_i' \partial_j ) \\
=~& \sum_{i < j} (-1)^{i+j} ~ (\partial_{j-1} \partial_i' + \partial_{j-1}' \partial_i ) ~+~ \sum_{i \geq j} (-1)^{i+j} ~ ( \partial_i \partial_j' + \partial_i' \partial_j ) \quad (\text{from } (\ref{presim-id})).
\end{align*}
Replacing $i$ by $n$ and $j$ by $m+1$ in the first summation, we get that it is opposite to the second summation. Hence they cancel with each other. This completes the proof.
\end{proof}

We are now ready to define a chain complex $\{ C^c, \partial^c \}$ associated to a compatible presimplicial vector space $(C, \Delta, \Delta')$ as follows. We define the $n$-th chain group $(C^c)_n$, for $n \geq 0,$ as

\begin{align*}
(C^c)_0 :=~& \{ c \in C_0 ~|~ c = \partial_0 x - \partial_1 x = \partial_0' x - \partial_1' x, \text{ for some } x \in C_1 \},\\
(C^c)_n :=~& \underbrace{C_n \oplus \cdots \oplus C_n}_{n \text{ times}}, \text{ for } n \geq 1,
\end{align*}
and the map $\partial^c : (C^c)n \rightarrow (C^c)_{n-1},$ for $n \geq 1$, given by
\begin{align}
(\partial^c)(x) :=~& \partial (x) = \partial' (x), \text{ for } x \in C_1, \label{boun1} \\
(\partial^c) (x_1, \ldots, x_n) :=~& (  \partial x_1 + \partial' x_2, \partial x_2 + \partial' x_3, \ldots, \partial x_{n-1} + \partial' x_n ), \label{boun2}
\end{align}
for $(x_1, \ldots, x_n) \in (C^c)_n$. The map $\partial^c$ can be understood by the following diagram:

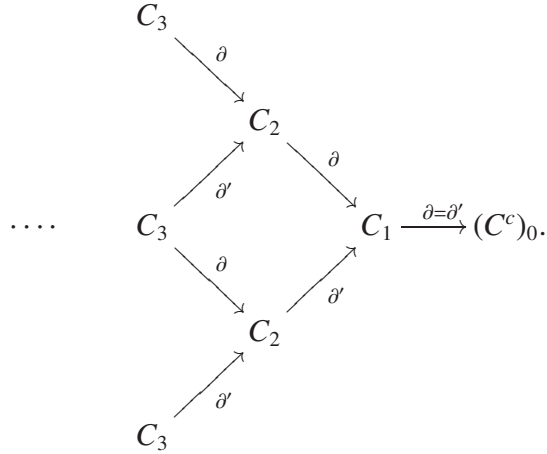
\begin{align*}
\xymatrix{
  & C_3 \ar[rd]^{\partial} & & & \\
  & & C_2 \ar[rd]^{\partial} & & \\
\cdots \cdot  & C_3 \ar[rd]^{\partial} \ar[ru]_{\partial'} & & C_1 \ar[r]^{\partial = \partial'} & (C^c)_0.\\
  & & C_2 \ar[ru]_{\partial'} & & \\
  & C_3 \ar[ru]_{\partial'}& & &
 }
\end{align*}

It is easy to verify using Proposition \ref{comp-pre-pro} that $(\partial^c)^2 =0$. In other words, $\{ C^c, \partial^c \}$ is a chain complex. The corresponding homology groups are called the homology associated with the compatible presimplicial vector space $(C, \Delta, \Delta').$

\subsection{Homology of compatible algebras}

Let $A$ be an associative algebra and $M$ be an $A$-bimodule. Consider the collection $C= \{ C_n (A, M) \}_{n \geq 0}$ of vector spaces given by
\begin{align*}
C_n (A, M)  := M \otimes A^{\otimes n}, ~ \text{ for } n \geq 0.
\end{align*}
It forms a presimplicial vector space with the collection of face maps $\Delta = \{ \partial_i : C_n (A,M) \rightarrow C_{n-1}(A,M) ~|~ i =0, 1, \ldots, n \}_{n \geq v1}$ given by
\begin{align*}
\partial_0 (m \otimes a_1 \otimes \cdots \otimes a_n) =~& m \cdot a_1 \otimes a_2 \otimes \cdots \otimes a_n,\\
\partial_i (m \otimes a_1 \otimes \cdots \otimes a_n) =~& m \otimes a_1 \otimes a_i \cdot a_{i+1} \otimes \cdots \otimes a_n, ~~~ 1 \leq i \leq n-1\\
\partial_n (m \otimes a_1 \otimes \cdots \otimes a_n) =~& a_n \cdot m \otimes a_1 \otimes \cdots \otimes a_{n-1}.
\end{align*}

The induced differential $\partial : C_n (A,M) \rightarrow C_{n-1} (A,M),$ for $n \geq 1$, is the Hochschild differential for the homology of $A$ with coefficients in the $A$-bimodule $M$.

\medskip

Next, let $A = (A, \mu_1, \mu_2)$ be a compatible associative algebra and $M$ be a  compatible $A$-bimodule. Then it follows from the above observation that $(C, \Delta)$ and $(C, \Delta')$ are presimplicial vector spaces, where $\Delta$ (resp. $\Delta'$) is a collection of face maps on $C$ induced by $(A, \mu_1)$-bimodule $(M,l_1, r_1)$ ~(resp. $(A, \mu_2)$-bimodule $(M,l_2, r_2)$ ).

\begin{pro}
With the above notations $(C, \Delta, \Delta')$ is a compatible presimplicial vector space.
\end{pro}

\begin{proof}
We only need to check the compatibility conditions (\ref{presim-id}). Let $0 \leq i < j \leq n$. In particular, if $0 < i < j-1$, then
\begin{align*}
&(\partial_i \partial_j' + \partial_i' \partial_j) ( m \otimes a_1 \otimes \cdots \otimes a_n ) \\
&= \partial_i ( m \otimes a_1 \otimes \cdots \otimes a_j \cdot_2 a_{j+1} \otimes \cdots \otimes a_n) ~+~ \partial_i' (m \otimes a_1 \otimes \cdots \otimes a_j \cdot_1 a_{j+1} \otimes \cdots \otimes a_n)  \\
&= m \otimes a_1 \otimes \cdots \otimes a_i \cdot_1 a_{i+1} \otimes \cdots \otimes \cdots \otimes a_j \cdot_2 a_{j+1} \otimes \cdots \otimes a_n \\
&\quad ~+~ m \otimes a_1 \otimes \cdots \otimes a_i \cdot_2 a_{i+1} \otimes \cdots \otimes \cdots \otimes a_j \cdot_1 a_{j+1} \otimes \cdots \otimes a_n \\
&= \partial_{j-1}' (m \otimes a_1 \otimes \cdots \otimes a_i \cdot_1 a_{i+1} \otimes \cdots \otimes a_n ) ~+~ \partial_{j-1} (m \otimes a_1 \otimes \cdots \otimes a_i \cdot_2 a_{i+1} \otimes \cdots \otimes a_n ) \\
&= (\partial_{j-1}' \partial_i + \partial_{j-1} \partial_i' ) ( m \otimes a_1 \otimes \cdots \otimes a_n ).
\end{align*}
In fact, for various choices of $i$ and $j$ satisfying $0 \leq i < j \leq n$, one can similarly show that $\partial_i \partial_j' + \partial_i' \partial_j = \partial_{j-1} \partial_i' + \partial_{j-1}' \partial_i$. This completes the proof.
\end{proof}

The above proposition suggests us to construct a chain complex associated to the compatible presimplicial vector space $(C, \Delta, \Delta')$.  More precisely, the $n$-th chain group $C_n^c (A, M)$, for $n \geq 0$, is given by
\begin{align*}
C_0^c (A, M) :=~& \{ m \in M ~|~ m= m' \cdot_1 a' - a' \cdot_1 m' =  m' \cdot_2 a' - a' \cdot_2 m', \text{ for some } m' \otimes a' \in M \otimes A \}, \\
C_n^c (A, M) :=~& \underbrace{ C_n (A, M) \oplus \cdots \oplus C_n (A, M)}_{n \text{ times}}, ~\text{ for } n \geq 1.
\end{align*}
The differential $\partial^c : C_n^c (A, M) \rightarrow C_{n-1}^c (A, M)$, for $n \geq 1$, is given by the formula (\ref{boun1}) and (\ref{boun2}), where $\partial$ (resp. $\partial'$) is the Hochschild boundary operator for the algebra $(A, \mu_1)$ with coefficients in the bimodule $(M, l_1, r_1)$ (resp. for the algebra $(A, \mu_2)$ with coefficients in the bimodule $(M, l_2, r_2)$). The corresponding homology groups are called the homology of the compatible associative algebra $A$ with coefficients in the compatible $A$-bimodule $M$, and they are denoted by $H_\ast^c (A,M).$.

\medskip

Let $(A, \mu_1, \mu_2)$ be an unital and commutative compatible associative algebra. We denote by $\Omega^1_{A | \mu_1 + \mu_2}$ the space generated by $\mathbb{K}$-linear symbols of the form $da$, for $a \in A$, subject to the relations
\begin{align*}
d ( a \cdot_1 b ) = a \cdot_1 db + da \cdot_1 b ~~ \text{ and } ~~ d ( a \cdot_2 b ) = a \cdot_2 db + da \cdot_2 b, \text{ for } a, b \in A.
\end{align*}
The space $\Omega^1_{A | \mu_1 + \mu_2}$ is a left module over the associative algebra $(A, \mu_1 + \mu_2)$. The elements of $\Omega^1_{A | \mu_1 + \mu_2}$ are called K\"{a}hler differentials.

\begin{pro}
Let $(A, \mu_1, \mu_2)$ be an unital and commutative compatible associative algebra. Then there is a canonical isomorphism $H_1^c (A,A) \cong \Omega^1_{A | \mu_1 + \mu_2}$ as $(A, \mu_1 + \mu_2)$-left module.
\end{pro}

\begin{proof}
The first homology group $H_1^c (A,A)$ is the quotient of $A \otimes A$ by the relation
\begin{align*}
a \cdot_1 b \otimes c - a \otimes b \cdot_1 c + c \cdot_1 a \otimes b + a \cdot_2 b \otimes c - a \otimes b \cdot_2 c + c \cdot_2 a \otimes b = 0.
\end{align*}
First observe that the map $A \otimes H_1^c (A,A) \rightarrow H_1^c (A,A), ~ a \otimes [b \otimes c] \mapsto a \cdot_1 b \otimes c + a \cdot_2 b \otimes c$, is a well-defined left $(A, \mu_1 + \mu_2)$-module structure on $H_1^c (A,A)$.

Next, we define a map $H^c_1 (A, A) \rightarrow \Omega^1_{A | \mu_1 + \mu_2}$ by $[a \otimes b] \mapsto a \cdot_1 db + a \cdot_2 db$. It is easy to verify that this map is well-defined and invertible. The inverse map is given by $(a \cdot_1 db + a \cdot_2 db ) \mapsto [a \otimes b]$. One can see that $a \otimes b$ is a $1$-cycle as both $\mu_1$ and $\mu_2$ are commutative. Hence the proof.
\end{proof}

\begin{rmk}
Our method of constructing the homology of compatible associative algebras can be easily generalized to compatible Lie algebras. Moreover, the standard skew-symmetrization leads to a morphism from the homology of a compatible associative algebra to the homology of the corresponding compatible Lie algebra. In forthcoming papers, we will come back with more properties of (co)homology of compatible associative algebras.
\end{rmk}

\section{Further discussions}\label{section:further}

In this paper, we introduce (co)homology of a compatible associative algebra $A = (A, \mu_1, \mu_2)$ with coefficients in a compatible $A$-bimodule. As applications of cohomology, we study extensions and deformations of $A$. Note that our (co)homology is not the same or combination of the Hochschild (co)homologies of $(A, \mu_1)$ and $(A, \mu_2)$. Here we collect some further questions regarding this new (co)homology theory.

\medskip

(I) {\bf Gerstenhaber structure on the cohomology.} In Theorem \ref{thm-gla-coho}, we show that the shifted space $H^{\ast +1}_c (A, A)$ of the cohomology of a compatible associative algebra $A$ with coefficients in itself carries a graded Lie bracket. One may now ask the following question: Is there any associative cup-product $\smile$ on the cohomology $H^{\ast}_c (A,A)$ which together with the graded Lie bracket makes $H^{\ast}_c (A,A)$ into a Gerstenhaber algebra? In \cite{das-gers}, we find an affirmative answer to this question in a more general context. Specifically, given a nonsymmetric operad with two compatible multiplications, we first construct a new cohomology generalizing our cohomology of compatible associative algebras. Then we show that this new cohomology can be seen as the induced cohomology of another multiplicative nonsymmetric operad. Hence by a result of Gerstenhaber and Voronov \cite{gers-voro}, the cohomology carries a Gerstenhaber structure. In particular, we explicitly write down the cup-product on the cohomology $H^{\ast}_c (A,A)$.

\medskip

(II) {\bf Compatible Rota-Baxter operators and compatible dendriform algebras.} Dendriform algebras was introduced in \cite{loday-di} as Koszul dual of associative dialgebras. They are certain splitting of associative algebras and arise naturally from shuffle algebras, planar binary trees and Rota-Baxter operators. Note that Rota-Baxter operators are a noncommutative analogue of Poisson structures \cite{uchino}. Motivated from the study of compatible Poisson structures in geometry, in a forthcoming paper \cite{das-guo}, we study compatible dendriform algebras and compatible Rota-Baxter operators from cohomological points of view and find relations with the results of the present paper.

\medskip

(III)  {\bf Cyclic (co)homology of compatible associative algebras.} The notion of cyclic (co)homology of an associative algebra generalizes the de Rham (co)homology of manifolds. Note that the cyclic (co)chain complexes (also called Connes complexes) are obtained from Hochschild (co)chain complexes modulo the actions of cyclic groups. In a future project, we aim to explore the cyclic (co)homology theory of compatible associative algebras and find its differential geometric significance.

\medskip

\noindent {\bf Acknowledgements.} The research of A. Das is supported by the postdoctoral fellowship of Indian Institute of Technology (IIT) Kanpur.


\end{document}